\documentclass[12pt, a4paper]{amsart}
\pdfoutput=1 
\usepackage[dvips]{graphicx}
\usepackage{amssymb, amstext, amscd, amsmath, color}
\usepackage{epstopdf}
\usepackage{tikz,mathdots,enumerate}
\usepackage{pgfplots}
\usepackage{url}

\usetikzlibrary{arrows}

\definecolor{cof}{RGB}{219,144,71}
\definecolor{pur}{RGB}{186,146,162}
\definecolor{greeo}{RGB}{91,173,69}
\definecolor{greet}{RGB}{52,111,72}

\usepackage{hhline}

\begin{document}

\newtheorem{theorem}{Theorem}[section]
\newtheorem{corollary}[theorem]{Corollary}
\newtheorem{proposition}[theorem]{Proposition}
\newtheorem{lemma}[theorem]{Lemma}

\theoremstyle{definition}
\newtheorem{remark}[theorem]{Remark}
\newtheorem{definition}[theorem]{Definition}
\newtheorem{example}[theorem]{Example}
\newtheorem{conjecture}[theorem]{Conjecture}

\newcommand{\FFock}{\mathcal{F}}
\newcommand{\kil}{\mathsf{k}}
\newcommand{\Hil}{\mathsf{H}}
\newcommand{\hil}{\mathsf{h}}
\newcommand{\Kil}{\mathsf{K}}
\newcommand{\Real}{\mathbb{R}}
\newcommand{\Rplus}{\Real_+}

\newcommand{\bC}{{\mathbb{C}}}
\newcommand{\bD}{{\mathbb{D}}}
\newcommand{\bN}{{\mathbb{N}}}
\newcommand{\bQ}{{\mathbb{Q}}}
\newcommand{\bR}{{\mathbb{R}}}
\newcommand{\bT}{{\mathbb{T}}}
\newcommand{\bX}{{\mathbb{X}}}
\newcommand{\bZ}{{\mathbb{Z}}}
\newcommand{\bH}{{\mathbb{H}}}
\newcommand{\BH}{{\B(\H)}}
\newcommand{\bsl}{\setminus}
\newcommand{\ca}{\mathrm{C}^*}
\newcommand{\cstar}{\mathrm{C}^*}
\newcommand{\cenv}{\mathrm{C}^*_{\text{env}}}
\newcommand{\rip}{\rangle}
\newcommand{\ol}{\overline}
\newcommand{\td}{\widetilde}
\newcommand{\wh}{\widehat}
\newcommand{\sot}{\textsc{sot}}
\newcommand{\wot}{\textsc{wot}}
\newcommand{\wotclos}[1]{\ol{#1}^{\textsc{wot}}}
 \newcommand{\A}{{\mathcal{A}}}
 \newcommand{\B}{{\mathcal{B}}}
 \newcommand{\C}{{\mathcal{C}}}
 \newcommand{\D}{{\mathcal{D}}}
 \newcommand{\E}{{\mathcal{E}}}
 \newcommand{\F}{{\mathcal{F}}}
 \newcommand{\G}{{\mathcal{G}}}
\renewcommand{\H}{{\mathcal{H}}}
 \newcommand{\I}{{\mathcal{I}}}
 \newcommand{\J}{{\mathcal{J}}}
 \newcommand{\K}{{\mathcal{K}}}
\renewcommand{\L}{{\mathcal{L}}}
 \newcommand{\M}{{\mathcal{M}}}
 \newcommand{\N}{{\mathcal{N}}}
\renewcommand{\O}{{\mathcal{O}}}
\renewcommand{\P}{{\mathcal{P}}}
 \newcommand{\Q}{{\mathcal{Q}}}
 \newcommand{\R}{{\mathcal{R}}}
\renewcommand{\S}{{\mathcal{S}}}
 \newcommand{\T}{{\mathcal{T}}}
 \newcommand{\U}{{\mathcal{U}}}
 \newcommand{\V}{{\mathcal{V}}}
 \newcommand{\W}{{\mathcal{W}}}
 \newcommand{\X}{{\mathcal{X}}}
 \newcommand{\Y}{{\mathcal{Y}}}
 \newcommand{\Z}{{\mathcal{Z}}}

\newcommand{\supp}{\operatorname{supp}}
\newcommand{\conv}{\operatorname{conv}}
\newcommand{\cone}{\operatorname{cone}}
\newcommand{\vspan}{\operatorname{span}}
\newcommand{\proj}{\operatorname{proj}}
\newcommand{\sgn}{\operatorname{sgn}}
\newcommand{\rank}{\operatorname{rank}}
\newcommand{\Isom}{\operatorname{Isom}}
\newcommand{\qIsom}{\operatorname{q-Isom}}
\newcommand{\Cknet}{{\mathcal{C}_{\text{knet}}}}
\newcommand{\Ckag}{{\mathcal{C}_{\text{kag}}}}
\newcommand{\rind}{\operatorname{r-ind}}
\newcommand{\lind}{\operatorname{r-ind}}
\newcommand{\ind}{\operatorname{ind}}
\newcommand{\coker}{\operatorname{coker}}
\newcommand{\Aut}{\operatorname{Aut}}
\newcommand{\Hom}{\operatorname{Hom}}
\newcommand{\GL}{\operatorname{GL}}
\newcommand{\tr}{\operatorname{tr}}

\newcommand{\eqnwithbr}[2]{%
\refstepcounter{equation}
\begin{trivlist}
\item[]#1 \hfill $\displaystyle #2$ \hfill (\theequation)
\end{trivlist}}

\setcounter{tocdepth}{1}

\title[Symmetric isostatic frameworks with $\ell^1$ or $\ell^\infty$ distance constraints]{Symmetric isostatic frameworks with $\ell^1$ or $\ell^\infty$ distance constraints}

\author[D. Kitson and  B. Schulze]{Derek Kitson  and  Bernd Schulze}
 \thanks{The first named author is supported by EPSRC grant  EP/J008648/1.}
\thanks{The second named author is supported by EPSRC grant  EP/M013642/1.}
\email{d.kitson@lancaster.ac.uk, b.schulze@lancaster.ac.uk}
\address{Dept.\ Math.\ Stats.\\ Lancaster University\\
Lancaster LA1 4YF \\U.K. }


\begin{abstract}
Combinatorial characterisations of minimal rigidity are obtained for  symmetric $2$-dimensional bar-joint frameworks with either $\ell^1$ or $\ell^\infty$ distance constraints. The characterisations are expressed in terms of symmetric tree packings and the number of edges fixed by the symmetry operations. The proof uses new Henneberg-type inductive construction schemes. 
\end{abstract}

\maketitle


\section{Introduction}
A fundamental problem in geometric rigidity theory is to find combinatorial characterisations of graphs which form rigid bar-joint frameworks for all generic realisations of the vertices in a given space. For the Euclidean plane, this problem was first solved by Laman's landmark result from 1970 (\cite{Lamanbib}) which characterises minimally rigid ({\em isostatic}) frameworks in terms of sparsity counts. Equivalent characterisations of generically rigid graphs in the plane have been obtained in terms of tree decompositions (see e.g. \cite{crapo,tay}) and using matroidal methods (\cite{lovyem}). 
In higher dimensions, generic rigidity has been characterised in terms of tree packings for body-bar, body-hinge and molecular frameworks (see \cite{jj,kattan,Tay84,tw,wwmatun}). An active research area is to consider the impact of symmetry on the rigidity of structures and various symmetric extensions of the aforementioned results have been established (see \cite{schulze,BS4,schtan,schtan2}).  For example, in \cite{schulze} and \cite{BS4} symmetric versions of Laman's theorem were obtained for the three-fold rotational symmetry group $\C_3$ and for the reflectional and half-turn rotational symmetry groups $\C_s$ and $\C_2$. However, the analogous questions for the remaining symmetry groups (i.e. the dihedral groups $\C_{2v}$ and $\C_{3v}$) remain open.

A natural problem which has drawn recent interest is to develop rigidity theory in the presence of non-Euclidean distance constraints (see for example \cite{gor-thu,kitson,sit-wil,sta}). 
In the case of the $\ell^1$ and $\ell^\infty$ norms, the only possible symmetry groups for a bar-joint framework in the plane are the reflection group $\C_s$, the rotation groups $\C_2$ and $\C_4$, and the dihedral group $\C_{2v}$. It is known that the existence of a (non-symmetric) isostatic placement of a graph in the plane is characterised by the existence of a spanning tree decomposition (see \cite{kit-pow}). Moreover, in \cite{kit-sch}, necessary counting conditions were obtained on the structure graph of an isostatic symmetric bar-joint framework for each of the possible symmetry groups. In this article, these results are combined together with new graph construction schemes to obtain characterisations for the existence of a symmetric isostatic placement of a graph in the plane. 
Section~\ref{sec:construction} presents the graph construction scheme
 (Theorem \ref{ConstructionScheme}) and in Section~\ref{sec:suffcon} the following characterisations are established. (See also Figure~\ref{fig:symisofw}.)

\begin{figure}[htp]
\begin{center}
\begin{tikzpicture}[very thick,scale=0.8]
\tikzstyle{every node}=[circle, draw=black, fill=white, inner sep=0pt, minimum width=5pt];

\path (0,0.7) node (p1)  {} ;
\path (-1.2,0) node (p2)  {} ;
\path (1.2,0) node (p3)  {} ;
\path (-1.3,-0.9) node (p4)  {} ;
\path (1.3,-0.9) node (p6)  {} ;

\draw (p1)  --  (p2);
\draw (p1)  --  (p3);
\draw[lightgray] (p1)  --  (p4);
\draw[lightgray] (p1)  --  (p6);
\draw[lightgray] (p2)  --  (p4);
\draw[lightgray] (p3)  --  (p6);
\draw (p4)  --  (p3);
\draw (p6)  --  (p2);
\draw[dashed, thick] (0,-1.3)--(0,1.3);
\node[rectangle, draw=white,fill=white](l1) at (0,-2) {(A)};
\end{tikzpicture} 
\hspace{1.2cm}
\begin{tikzpicture}[very thick,scale=0.8,rotate=-45]
\tikzstyle{every node}=[circle, draw=black, fill=white, inner sep=0pt, minimum width=5pt];
\path (-0.35,-0.8) node (p1)  {} ;
\path (0.35,-0.8) node (p2)  {} ;
\path (-1.6,0.1) node (p3)  {} ;
\path (1.6,0.1) node (p4)  {} ;
\path (-0.7,0.8) node (p5)  {} ;
\path (0.7,0.8) node (p6)  {} ;

\draw[lightgray] (p1)  --  (p5);
\draw[lightgray] (p1)  --  (p3);
\draw (p1)  --  (p6);
\draw[lightgray] (p2)  --  (p5);
\draw (p2)  --  (p4);
\draw (p2)  --  (p6);
\draw (p5)  --  (p3);
\draw (p3)  --  (p6);
\draw[lightgray] (p5)  --  (p4);
\draw[lightgray] (p4)  --  (p6);
\draw[dashed, thick] (0,-1.7)--(0,1.7);
\node[rectangle, draw=white,fill=white](l1) at (1.5,-1.3) {(B)};
\end{tikzpicture}
\hspace{1.2cm}
\begin{tikzpicture}[very thick,scale=0.8]
\tikzstyle{every node}=[circle, draw=black, fill=white, inner sep=0pt, minimum width=5pt];

\path (0,0) node (p1)  {} ;
\path (-0.9,-0.9) node (p2)  {} ;
\path (1,-0.7) node (p3)  {} ;
\path (0.9,0.9) node (p4)  {} ;
\path (-1,0.7) node (p5)  {} ;
\path (-1.6,0.5) node (p6)  {} ;
\path (1.6,-0.5) node (p7)  {} ;

\draw[lightgray] (p1)  --  (p2);
\draw (p1)  --  (p3);
\draw[lightgray] (p1)  --  (p4);
\draw (p1)  --  (p5);
\draw (p2)  --  (p3);
\draw[lightgray] (p3)  --  (p4);
\draw (p4)  --  (p5);
\draw[lightgray] (p5)  --  (p2);

\draw[lightgray] (p6)  --  (p2);
\draw (p6)  --  (p5);
\draw (p7)  --  (p3);
\draw[lightgray] (p7)  --  (p4);
\node[rectangle, draw=white,fill=white](l1) at (0,-2) {(C)};
\end{tikzpicture} 
\hspace{1.2cm}
\begin{tikzpicture}[very thick,scale=0.8]
\tikzstyle{every node}=[circle, draw=black, fill=white, inner sep=0pt, minimum width=5pt];

 \node (p1) at (52:0.8cm) {};
 \node (p2) at (142:0.8cm) {};
 \node (p3) at (232:0.8cm) {};
 \node (p4) at (322:0.8cm) {};

 \node (v1) at (60:1.5cm) {};
 \node (v2) at (150:1.5cm) {};
 \node (v3) at (240:1.5cm) {};
 \node (v4) at (330:1.5cm) {};

\draw (p1)  --  (p2);
\draw[lightgray] (p2)  --  (p3);
\draw (p4)  --  (p3);
\draw[lightgray] (p1)  --  (p4);
\draw[lightgray] (p1)  --  (p3);
\draw (p4)  --  (p2);

\draw (v1)  --  (v2);
\draw[lightgray] (v2)  --  (v3);
\draw (v4)  --  (v3);
\draw[lightgray] (v1)  --  (v4);

\draw[lightgray] (v1)  --  (p1);
\draw (v2)  --  (p2);
\draw (v4)  --  (p4);
\draw[lightgray] (v3)  --  (p3);
\node[rectangle, draw=white,fill=white](l1) at (0,-2) {(D)};
\end{tikzpicture} 
     \end{center}
\vspace{-0.3cm}
     \caption{Examples of symmetric isostatic frameworks in the plane with $\ell^\infty$ distance constraints, illustrating the four statements in Theorem~\ref{CsThm}.}
\label{fig:symisofw}
\end{figure}
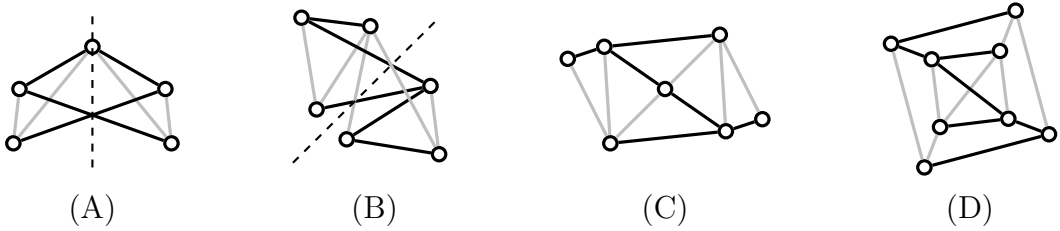

\begin{theorem}
\label{CsThm}
Let $\|\cdot\|_\P$ be a norm on $\bR^2$ for which the unit ball $\P$ is a quadrilateral and let $G$ be a finite simple graph with a group action $\theta:\bZ_n\to \Aut(G)$ where $\bZ_n=\langle \gamma\rangle$, $n\in\{2,4\}$.

\begin{enumerate}[(A)]
	\item 
The following statements are equivalent.
\begin{enumerate}[(i)]
\item There exists a representation $\tau:\bZ_2\to \GL(\bR^2)$ and a point $p$ such that the bar-joint framework $(G,p)$ is well-positioned and isostatic  in $(\bR^2,\|\cdot\|_\P)$ and $\C_s$-symmetric with respect to $\theta$ and $\tau$, where the symmetry operation $\gamma$ is a  reflection which preserves the facets of $\P$.
\item  $G$ is expressible as a union of two edge-disjoint spanning trees,  both of which are $\bZ_2$-symmetric with respect to $\theta$, and no edge of $G$ is fixed by $\gamma$.
\end{enumerate}

\item 
The following statements are equivalent.
\begin{enumerate}[(i)]
\item There exists a representation $\tau:\bZ_2\to \GL(\bR^2)$ and a point $p$ such that the bar-joint framework $(G,p)$ is well-positioned and isostatic  in $(\bR^2,\|\cdot\|_\P)$ and $\C_s$-symmetric with respect to $\theta$ and $\tau$, where the symmetry operation $\gamma$ is a  reflection which swaps the facets of $\P$.
\item  $G$ is expressible as a union of two edge-disjoint spanning trees $G_1$ and $G_2$,  such that $G_1=\gamma(G_2)$ and $G_2=\gamma(G_1)$, and no edge of $G$ is fixed by $\gamma$.
\end{enumerate}

\item
The following statements are equivalent.
\begin{enumerate}[(i)]
\item There exists a representation $\tau:\bZ_2\to \GL(\bR^2)$ and a point $p$ such that the bar-joint framework $(G,p)$ is well-positioned and isostatic  in $(\bR^2,\|\cdot\|_\P)$ and $\C_2$-symmetric with respect to $\theta$ and $\tau$, where the symmetry operation $\gamma$ is a half-turn rotation.
\item  $G$ is expressible as a union of two edge-disjoint spanning trees,  both of which are $\bZ_2$-symmetric with respect to $\theta$, and either no edge or two edges of $G$ are fixed by $\gamma$. 
\end{enumerate}

\item The following statements are equivalent.
\begin{enumerate}[(i)]
\item There exists a representation $\tau:\bZ_4\to \GL(\bR^2)$ and a point $p$ such that the bar-joint framework $(G,p)$ is well-positioned and isostatic  in $(\bR^2,\|\cdot\|_\P)$ and $\C_4$-symmetric with respect to $\theta$ and $\tau$, where the symmetry operation $\gamma$ is a quarter-turn rotation.
\item  $G$ is expressible as a union of two edge-disjoint spanning trees $G_1$ and $G_2$, 
such that $G_1=\gamma(G_2)$ and $G_2=\gamma(G_1)$ and either no edge or two edges of $G$ are fixed by $\gamma^2$.
\end{enumerate}
\end{enumerate}
\end{theorem}



Theorem \ref{CsThm} provides the first combinatorial characterisations for symmetric isostatic frameworks in a non-Euclidean normed linear space. 
The results highlight an interplay between the symmetry operations of the framework, the unit ball in the normed space and spanning trees in the underlying graph. 
The statement of the theorem, and its proof, illustrate a sharp contrast with the corresponding Euclidean rigidity theory \cite{schulze,BS4}.
For the proof of Theorem \ref{CsThm}, several new graph moves are introduced together with rigidity-preserving geometric placements which may be of independent interest.  
The results have potential practical applications for non-Euclidean constraint systems (eg. polygonal packings, graph realizability, autonomous agents and CAD) and also offer a new perspective on symmetric tree packings in simple graphs. 
 It is an open problem to obtain a corresponding theorem for the  symmetry group $\C_{2v}$ and for other norms and symmetry groups. 


\section{Construction scheme for $\bZ_2$-symmetric graphs.}
\label{sec:construction}
An action of a group $\Gamma$ on a simple graph $G$ is a group homomorphism
$\theta:\Gamma\to\Aut(G)$. Here $\Aut(G)$ is the automorphism group of $G$ which consists of permutations of the vertices $\pi:V\to V$ such that $\pi(v)\pi(w)$ is an edge of $G$ if and only if $vw$ is an edge of $G$. If such an action exists then $G$ is said to be {\em $\Gamma$-symmetric with respect to the action $\theta$}
(or simply {\em $\Gamma$-symmetric} when the action is clear). If $H$ is a graph such that the vertex set of $H$ is a subset of the vertex set of $G$ and this vertex set is invariant under the permutation $\theta(\gamma)$ for each $\gamma\in\Gamma$ then $\theta$ induces an action of $\Gamma$ on $H$. For convenience this induced action is also denoted by $\theta$. 
If the action of $\theta$ is clear from the context then $\gamma v$ is used instead of $\theta(\gamma)v$ for each vertex $v\in V(G)$
and $\gamma(vw)$ instead of $(\theta(\gamma)v)(\theta(\gamma)w)$ for each edge $vw\in E(G)$.
A vertex of $G$ is fixed by $\gamma$ if $\gamma v=v$ and an edge $vw$ is fixed by $\gamma$ if $\gamma(vw)=vw$. 

\begin{definition}
Let $\theta:\bZ_2\to \Aut(G)$ be an action of the group $\bZ_2=\langle s\rangle$ on a finite simple graph $G$.
The pair $(G,\theta)$ is {\em admissible}  if it has the following properties.
\begin{enumerate}[(i)]
	\item $G$ is expressible as a union of two edge-disjoint spanning trees, 
		both of which are $\bZ_2$-symmetric with respect to $\theta$, and,
	\item no edge of $G$ is fixed by $s$.
\end{enumerate}
\end{definition}

In this section a construction scheme will be established for the class of admissible pairs $(G,\theta)$. This scheme is comprised of four graph extensions and a base graph $W_5$. The following elementary facts will be required. 

\begin{lemma}
\label{FixedVertexLemma} 
If  $(G,\theta)$ is an admissible pair then,
\begin{enumerate}[(i)]
\item there exists exactly one vertex $v_0$ in $G$ which is fixed by $s$, and,
\item  the unique fixed vertex $v_0$ has even degree and degree at least $4$.
\end{enumerate}
\end{lemma}

\proof 
The graph $G$ is expressible as a union of two edge-disjoint $\bZ_2$-symmetric spanning trees $G_1$ and $G_2$. 
If two distinct vertices, $v_0$ and $v_1$,  are fixed by  $s$
then $G_1$ must contain simple paths $P$ and $s(P)$ joining $v_0$ to $v_1$. 
Since no edge of $G$ is fixed by $s$, $P\not=s(P)$.
This is a contradiction and so there is at most one vertex in $G$ which is fixed by $s$.
Since $s$ acts as an involution on the edges of $G_1$ with no fixed points, $|E(G_1)|$ is even. In particular, $|V(G)|$ is odd and, since $s$ acts as an involution on $V(G)$, at least one vertex in $G$ must be fixed by $s$.
This proves the first statement.
For the second statement, note that $s$ acts as an involution with no fixed points on the set of edges adjacent to $v_0$. Thus, $v_0$ has even degree in $G$.
Also, $v_0$ is adjacent to some distinct edges $e,se\in E(G_1)$ and  $f,sf\in E(G_2)$ and so $v_0$ has degree at least $4$ in $G$.
\endproof

For a vertex $v$ in $G$, the set consisting of all vertices which are adjacent to $v$ is denoted $N(v)$. 

\begin{definition}
Let $\theta:\bZ_2\to \Aut(G)$ be an action of the group $\bZ_2=\langle s\rangle$ on $G$ and let $v\in V(G)$.
The {\em symmetric neighbourhood} of $v$, denoted $\S(v)$, is the subgraph of $G$ induced by the vertex set $N(v)\cup N(sv)\cup\{v,sv\}$.
\end{definition}

Three of the four graph extensions in the construction scheme for admissible pairs are determined by identifying the possible symmetric neighbourhoods of a $3$-valent vertex $v$ in $G$. This approach is coordinated by separating the symmetric neighbourhoods $\S(v)$ into four types depending on the cardinality of $N(v)\cap N(sv)$.

\begin{lemma}
\label{NbdLemma}
Let $(G,\theta)$ be an admissible pair.
If $v$ is a $3$-valent vertex in $G$ then $N(v)$ and $N(sv)$ are either, 
\begin{enumerate}[(i)]
	\item disjoint, 
	\item intersect in one vertex, which must be the fixed vertex, 
	\item intersect in two vertices, neither of which are the fixed vertex, or,
	\item intersect in three vertices, one of which is the fixed vertex.
\end{enumerate}
\end{lemma}

\proof
The induced action of  $s$ on $N(v)\cap N(sv)$ is an involution. 
If $N(v)$ and $N(sv)$ intersect in one vertex, then this vertex is clearly fixed by $s$.
If $N(v)$ and $N(sv)$ intersect in two vertices then neither, or, both of these vertices are fixed by $s$. 
By Lemma \ref{FixedVertexLemma}, $G$ contains only one vertex which is fixed by $s$, and so neither vertex of $N(v)\cap N(sv)$ is fixed by $s$.
If $N(v)$ and $N(sv)$ intersect in three vertices then $s$ acts as a transposition on $N(v)\cap N(sv)$  leaving one vertex fixed.
\endproof

The {\em wheel graph} $W_5$ is the graph with five distinct vertices $v_0,v_1,\ldots,v_{4}$ and an edge set consisting of the $4$-cycle $v_1v_2,\, v_2v_3,\,v_3v_4,\, v_{4}v_{1}$ together with four {\em spokes} $\{v_0v_j:j=1,2,3,4\}$. 
The wheel graph $W_5$ will form the base graph in the construction scheme for admissible pairs. 
(See also Figure \ref{fig:symisofw}(A)).

\begin{example}
\label{WheelGraph}
Let $W_5$ be the wheel graph  and define an action $\theta^*:\bZ_2\to \Aut(W_5)$  by setting $sv_1=v_3$, $sv_2=v_4$ and $sv_0=v_0$.
Then $W_5$ is expressible as an edge-disjoint union of two $\bZ_2$-symmetric spanning trees $T_1$ and $T_2$ with edge sets $E(T_1)=\{v_2v_3,\,v_3v_0,\,v_0v_1,\,v_1v_4\}$ and  $E(T_2)=\{v_1v_2,\,v_2v_0,\,v_0v_4,\,v_4v_3\}$. 
Note that $W_5$ has no edges which are fixed by $s$ and so $(W_5,\theta^*)$ is an admissible pair. Also, $\theta^*$ is the only action of $\bZ_2$ on $W_5$ for which no edge of $W_5$ is fixed by $s$.
\end{example}

\subsection{$(\bZ_2,\theta)$ $0$-extensions.}

The first graph extension in the construction scheme for admissible pairs involves $2$-valent vertices. (See  Fig.~\ref{1Ext} for an illustration).

\begin{definition}
\label{Cs1extension}
Let $G$ be a simple graph which is $\bZ_2$-symmetric with respect to $\theta$. 
Suppose there exists a graph $H$ and distinct vertices $v_1,v_2\in V(H)$ with the following properties.
\begin{enumerate}[(i)]
\item $V(G)=V(H)\cup \{v,sv\}$ where $v,sv\notin V(H)$ and $v\not=sv$.
\item $H$ is $\bZ_2$-symmetric with respect to $\theta$.
\item $E(G)=E(H)\cup \big\{vv_{1},vv_{2},s(vv_{1}),s(vv_{2})\big\}$.
\end{enumerate}
Then  $G$ is said to be obtained from $H$ by applying a $(\bZ_2,\theta)$ \emph{$0$-extension} (on the vertices $v_1,v_2$). 
\end{definition}

\begin{figure}[htp]
\begin{center}
\begin{tikzpicture}[very thick,scale=0.9]
\tikzstyle{every node}=[circle, draw=black, fill=white, inner sep=0pt, minimum width=5pt];
\filldraw[fill=black!03!white, draw=black, thin, dashed](0,0)circle(1.3cm);
\node[rectangle, draw=black!03!white,fill=black!03!white](l1) at (-0.7,-0.3) {$v_2$};
\node[rectangle, draw=black!03!white,fill=black!03!white](l2) at (0.6,-0.3) {$sv_2$};
\node[rectangle, draw=black!03!white,fill=black!03!white](l3) at (-0.5,0.73) {$v_1$};
\node[rectangle, draw=black!03!white,fill=black!03!white](l4) at (0.4,0.73) {$sv_1$};
\node (p1) at (-0.5,0.4) {};
\node (p2) at (0.5,0.4) {};
\node (p3) at (-0.72,-0.67) {};
\node (p4) at (0.72,-0.67) {};
\filldraw[fill=black!50!white, draw=black, thick]
    (2.45,0) -- (3.05,0) -- (3.05,-0.1) -- (3.25,0.05) -- (3.05,0.2) -- (3.05,0.1) -- (2.45,0.1) -- (2.45,0);
\end{tikzpicture}
\hspace{0.7cm}
\begin{tikzpicture}[very thick,scale=0.9]
\tikzstyle{every node}=[circle, draw=black, fill=white, inner sep=0pt, minimum width=5pt];
\filldraw[fill=black!03!white, draw=black, thin, dashed](0,0)circle(1.3cm);
\node[rectangle, draw=black!03!white,fill=black!03!white](l1) at (-0.65,-0.3) {$v_2$};
\node[rectangle, draw=black!03!white,fill=black!03!white](l2) at (0.55,-0.3) {$sv_2$};
\node[rectangle, draw=black!03!white,fill=black!03!white](l3) at (-0.5,0.73) {$v_1$};
\node[rectangle, draw=black!03!white,fill=black!03!white](l4) at (0.4,0.73) {$sv_1$};
\node (p1) at (-0.5,0.4) {};
\node (p2) at (0.5,0.4) {};
\node (p3) at (-0.72,-0.67) {};
\node (p4) at (0.72,-0.67) {};
\node (p5) at (-1.6,0) {};
\node (p6) at (1.6,0) {};
\node[rectangle, draw=white,fill=white](l4) at (-1.9,0) {$v$};
\node[rectangle, draw=white,fill=white](l4) at (2,0) {$sv$};
\draw(p5)--(p1);
\draw(p5)--(p3);
\draw(p6)--(p2);
\draw(p6)--(p4);
\end{tikzpicture}
\end{center}
\vspace{-0.2cm}
\caption{\emph{A $(\bZ_2,\theta)$ $0$-extension $G$ of a graph $H$.}}
\label{1Ext}
\end{figure}
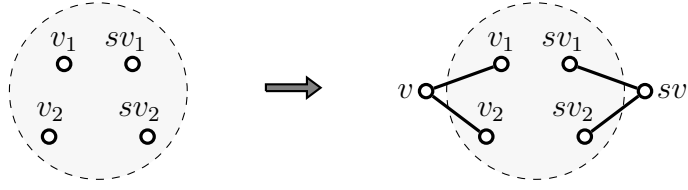

The proof that admissible pairs $(G,\theta)$ are constructible by graph extensions is by induction on the number of vertices of $G$.
To this end, it will be shown that if $|V(G)|\geq6$ then $G$ is obtained from a graph $H$ with fewer vertices by some graph extension in the construction scheme  and that $(H,\theta)$ is an admissible pair.  If $G$ contains a $2$-valent vertex then this is achieved by a $(\bZ_2,\theta)$ $0$-extension.  

\begin{lemma}
\label{Degree2VertexLemma}
Let $(G,\theta)$ be an admissible pair and let $v$ be a $2$-valent vertex in $G$.
Then  there exists a graph $H$ with the following properties. 
\begin{enumerate}[(i)]
\item $G$ is obtained from $H$ by applying a $(\bZ_2,\theta)$ $0$-extension.
\item $(H,\theta)$ is an admissible pair. 
\end{enumerate}
\end{lemma}

\proof
By Lemma \ref{FixedVertexLemma}, the unique fixed vertex $v_0$ in $G$ has degree at least $4$ and so $v\not=sv$.
Let $H=G\backslash\{v,sv\}$. 
Then $H$ is  $\bZ_2$-symmetric with respect to $\theta$ and $G$ is obtained from $H$ by applying a $(\bZ_2,\theta)$ $0$-extension on the vertices of $N(v)$. 
The graph $G$ may be expressed as a union of two edge-disjoint $\bZ_2$-symmetric spanning trees $G_1$ and $G_2$.
Let $H_1=G_1\cap H$ and $H_2=G_2\cap H$.
Then $H_1$ and $H_2$ are edge-disjoint $\bZ_2$-symmetric spanning trees in $H$  and $H$ is the union of $H_1$ and $H_2$. Also, $H$ has no edges which are fixed by $s$ and so $(H,\theta)$ is an admissible pair. 
\endproof

\subsection{$(\bZ_2,\theta)$ $1$-extensions}
The second graph extension in the construction scheme is a $(\bZ_2,\theta)$ $1$-extension. (See Fig.~\ref{2Ext}.)
This, and the remaining  graph extensions, apply to $3$-valent vertices. 

\begin{definition}
\label{Cs2extension}
Let $G$ be a simple graph which is $\bZ_2$-symmetric with respect to $\theta$. 
Suppose there exists a graph $H$, distinct vertices $v_{1},v_{2},v_{3}\in V(H)$ and an edge $e=v_1v_2\in E(H)$ with the following properties.
\begin{enumerate}[(i)]
\item $V(G)=V(H)\cup \{v,sv\}$ where $v,sv\notin V(H)$ and $v\not=sv$.
\item $H$ is $\bZ_2$-symmetric with respect to $\theta$  and $e\not=se$.
\item $E(G)=\big(E(H)\setminus \big\{e,se\big\}\big)\cup \big\{vv_{i},s(vv_{i})\,|\,\, i=1,2,3\big\}$.
\end{enumerate}
Then  $G$ is said to be obtained from $H$ by applying a $(\bZ_2,\theta)$ \emph{$1$-extension} (on the vertices $v_1,v_2,v_3$ and the edge $e$). 
\end{definition}

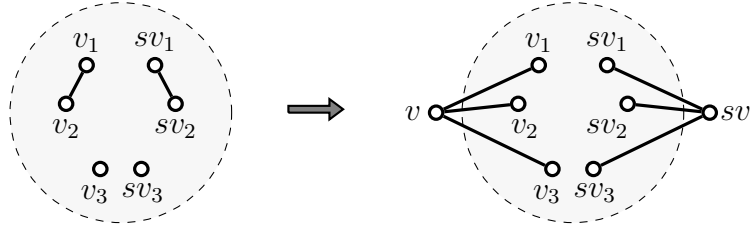
\begin{figure}[htp]
\begin{center}
\begin{tikzpicture}[very thick,scale=0.9]
\tikzstyle{every node}=[circle, draw=black, fill=white, inner sep=0pt, minimum width=5pt];
\filldraw[fill=black!03!white, draw=black, thin, dashed](0,0)circle(1.62cm);
\node[rectangle, draw=black!03!white,fill=black!03!white](l1) at (-0.5,1.03) {$v_1$};
\node[rectangle, draw=black!03!white,fill=black!03!white](l2) at (0.5,1.05) {$sv_1$};
\node[rectangle, draw=black!03!white,fill=black!03!white](l1) at (-0.8,-0.23) {$v_2$};
\node[rectangle, draw=black!03!white,fill=black!03!white](l2) at (0.8,-0.23) {$sv_2$};
\node[rectangle, draw=black!03!white,fill=black!03!white](l1) at (-0.38,-1.2) {$v_3$};
\node[rectangle, draw=black!03!white,fill=black!03!white](l2) at (0.32,-1.18) {$sv_3$};
\node (p1) at (-0.5,0.7) {};
\node (p2) at (0.5,0.7) {};
\node (p3) at (-0.8,0.13) {};
\node (p4) at (0.8,0.13) {};
\node (p5) at (-0.3,-0.83) {};
\node (p6) at (0.3,-0.83) {};
\draw(p1)--(p3);
\draw(p2)--(p4);
\filldraw[fill=black!50!white, draw=black, thick]
    (2.45,0) -- (3.05,0) -- (3.05,-0.1) -- (3.25,0.05) -- (3.05,0.2) -- (3.05,0.1) -- (2.45,0.1) -- (2.45,0);
\end{tikzpicture}
\hspace{0.5cm}
\begin{tikzpicture}[very thick,scale=0.9]
\tikzstyle{every node}=[circle, draw=black, fill=white, inner sep=0pt, minimum width=5pt];
\filldraw[fill=black!03!white, draw=black, thin, dashed](0,0)circle(1.62cm);
\node[rectangle, draw=black!03!white,fill=black!03!white](l1) at (-0.5,1.03) {$v_1$};
\node[rectangle, draw=black!03!white,fill=black!03!white](l2) at (0.5,1.05) {$sv_1$};
\node[rectangle, draw=black!03!white,fill=black!03!white](l1) at (-0.7,-0.23) {$v_2$};
\node[rectangle, draw=black!03!white,fill=black!03!white](l2) at (0.5,-0.23) {$sv_2$};
\node[rectangle, draw=black!03!white,fill=black!03!white](l1) at (-0.38,-1.2) {$v_3$};
\node[rectangle, draw=black!03!white,fill=black!03!white](l2) at (0.32,-1.18) {$sv_3$};
\node (p1) at (-0.5,0.7) {};
\node (p2) at (0.5,0.7) {};
\node (p3) at (-0.8,0.13) {};
\node (p4) at (0.8,0.13) {};
\node (p5) at (-0.3,-0.83) {};
\node (p6) at (0.3,-0.83) {};
\node (p7) at (-2,0) {};
\node (p8) at (2,0) {};
\node[rectangle, draw=white,fill=white](l1) at (-2.34,0) {$v$};
\node[rectangle, draw=white,fill=white](l2) at (2.38,0) {$sv$};
\draw(p7)--(p1);
\draw(p7)--(p3);
\draw(p7)--(p5);
\draw(p8)--(p2);
\draw(p8)--(p4);
\draw(p8)--(p6);
\end{tikzpicture} 
\end{center}
\vspace{-0.2cm}
\caption{\emph{A $(\bZ_2,\theta)$ $1$-extension $G$ of a graph $H$.}}
\label{2Ext}
\end{figure}

If $(G,\theta)$ is an admissible pair and if $G$ contains a $3$-valent vertex $v$ then it may not be the case that $G$ can be obtained from a graph $H$ by a $(\bZ_2,\theta)$ $1$-extension. 
The $(\bZ_2,\theta)$ $1$-extension is sufficient, however,  under the additional assumption that $N(v)$ and $N(sv)$ intersect in exactly two vertices.

\begin{lemma}
\label{2VertexLemma}
Let $(G,\theta)$ be an admissible pair and let $v$ be a $3$-valent vertex in $G$.
Suppose $N(v)$ and $N(sv)$ intersect in two vertices.
Then  there exists a graph $H$ with the following properties. 
\begin{enumerate}[(i)]
\item $G$ is obtained from $H$ by applying a $(\bZ_2,\theta)$ $1$-extension.
\item $(H,\theta)$ is an admissible pair. 
\end{enumerate}
\end{lemma}

\proof
By Lemma \ref{FixedVertexLemma}, $v$ is not fixed by $s$.
By Lemma \ref{NbdLemma}, $N(v)$ does not contain the fixed vertex $v_0$ and 
so $N(v)=\{v_1,sv_1,v_2\}$ for some $v_1,v_2\in V(G)\backslash\{v_0\}$.
The graph $G$ may be expressed as a union of two edge-disjoint $\bZ_2$-symmetric spanning trees $G_1$ and $G_2$. Thus $\S(v)$ appears in Fig. \ref{Degree3VertexC}  with edges of $G_1$ and $G_2$ indicated in black and gray respectively.
Note that either $vv_1\notin E(G_1)$ or $v(sv_1)\notin E(G_1)$ since, otherwise, $G_1$ would contain a cycle. The cases where $v_1v_2\notin E(G)$
and where $v_1v_2\in E(G)$ are considered separately.

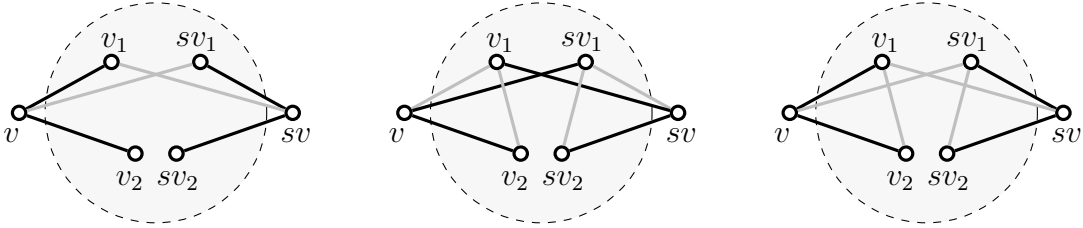
\begin{figure}[htp]
\centering
\begin{tabular}{c}

\begin{tikzpicture}[very thick,scale=0.9]
\tikzstyle{every node}=[circle, draw=black, fill=white, inner sep=0pt, minimum width=5pt];
\filldraw[fill=black!03!white, draw=black, thin, dashed](0,0)circle(1.62cm);
\node[rectangle, draw=black!03!white,fill=black!03!white](l1) at (-0.6,1.03) {$v_1$};
\node[rectangle, draw=black!03!white,fill=black!03!white](l2) at (0.6,1.05) {$sv_1$};
\node[rectangle, draw=black!03!white,fill=black!03!white](l1) at (-0.38,-1) {$v_2$};
\node[rectangle, draw=black!03!white,fill=black!03!white](l2) at (0.32,-1) {$sv_2$};
\node (p1) at (-0.65,0.75) {};
\node (p2) at (0.65,0.75) {};
\node (p5) at (-0.3,-0.6) {};
\node (p6) at (0.3,-0.6) {};
\node (p7) at (-2,0) {};
\node (p8) at (2,0) {};
\node[rectangle, draw=white,fill=white,below right](l1) at (-2.25,-0.2) {$v$};
\node[rectangle, draw=white,fill=white,below left](l2) at (2.3,-0.2) {$sv$};
\draw(p7)--(p1);
\draw[lightgray](p7)--(p2);
\draw(p7)--(p5);
\draw(p8)--(p2);
\draw[lightgray](p8)--(p1);
\draw(p8)--(p6);
\end{tikzpicture}
\hspace{0.7cm}
\begin{tikzpicture}[very thick,scale=0.9]
\tikzstyle{every node}=[circle, draw=black, fill=white, inner sep=0pt, minimum width=5pt];
\filldraw[fill=black!03!white, draw=black, thin, dashed](0,0)circle(1.62cm);
\node[rectangle, draw=black!03!white,fill=black!03!white](l1) at (-0.6,1.03) {$v_1$};
\node[rectangle, draw=black!03!white,fill=black!03!white](l2) at (0.6,1.05) {$sv_1$};
\node[rectangle, draw=black!03!white,fill=black!03!white](l1) at (-0.38,-1) {$v_2$};
\node[rectangle, draw=black!03!white,fill=black!03!white](l2) at (0.32,-1) {$sv_2$};
\node (p1) at (-0.65,0.75) {};
\node (p2) at (0.65,0.75) {};
\node (p5) at (-0.3,-0.6) {};
\node (p6) at (0.3,-0.6) {};
\node (p7) at (-2,0) {};
\node (p8) at (2,0) {};
\node[rectangle, draw=white,fill=white,below right](l1) at (-2.25,-0.2) {$v$};
\node[rectangle, draw=white,fill=white,below left](l2) at (2.3,-0.2) {$sv$};
\draw[lightgray](p7)--(p1);
\draw(p7)--(p2);
\draw(p7)--(p5);
\draw[lightgray](p8)--(p2);
\draw(p8)--(p1);
\draw(p8)--(p6);
\draw[lightgray](p1)--(p5);
\draw[lightgray](p2)--(p6);
\end{tikzpicture}
\hspace{0.7cm}
\begin{tikzpicture}[very thick,scale=0.9]
\tikzstyle{every node}=[circle, draw=black, fill=white, inner sep=0pt, minimum width=5pt];
\filldraw[fill=black!03!white, draw=black, thin, dashed](0,0)circle(1.62cm);
\node[rectangle, draw=black!03!white,fill=black!03!white](l1) at (-0.6,1.03) {$v_1$};
\node[rectangle, draw=black!03!white,fill=black!03!white](l2) at (0.6,1.05) {$sv_1$};
\node[rectangle, draw=black!03!white,fill=black!03!white](l1) at (-0.38,-1) {$v_2$};
\node[rectangle, draw=black!03!white,fill=black!03!white](l2) at (0.32,-1) {$sv_2$};
\node (p1) at (-0.65,0.75) {};
\node (p2) at (0.65,0.75) {};
\node (p5) at (-0.3,-0.6) {};
\node (p6) at (0.3,-0.6) {};
\node (p7) at (-2,0) {};
\node (p8) at (2,0) {};
\node[rectangle, draw=white,fill=white,below right](l1) at (-2.25,-0.2) {$v$};
\node[rectangle, draw=white,fill=white,below left](l2) at (2.3,-0.2) {$sv$};
\draw(p7)--(p1);
\draw[lightgray](p7)--(p2);
\draw(p7)--(p5);
\draw(p8)--(p2);
\draw[lightgray](p8)--(p1);
\draw(p8)--(p6);
\draw[lightgray](p5)--(p1);
\draw[lightgray](p6)--(p2);
\end{tikzpicture}
\end{tabular}
\vspace{-0.1cm}
\caption{$N(v)$ and $N(sv)$ intersect in two vertices.}
\label{Degree3VertexC}
\end{figure}

{\bf Case (1)}
Suppose $v_1v_2\notin E(G)$. 
By relabeling $v_1$ and $sv_1$, it may be assumed that the edges $g=vv_1$ and $h=vv_2$ are both contained in the same spanning tree, $G_1$ say, and so $\S(v)$ is the leftmost subgraph in Fig. \ref{Degree3VertexC}. 
Let $H$ be the graph obtained by adjoining the edges $e=v_1v_2$ and $se$ to $G\backslash \{v,sv\}$.
Then $G$ is obtained from $H$ by applying a $(\bZ_2,\theta)$ $1$-extension on the vertices $v_1$, $sv_1$,  $v_2$ and the edge $e$. 
Let $H_1=(G_1\cap H)\cup\{e,se\}$ and $H_2=G_2\cap H$.
Then $H_1$ and $H_2$ are edge-disjoint $\bZ_2$-symmetric subgraphs in $H$.  
It is clear that $H_2$ is a spanning tree in $H$.
To see that $H_1$ is a spanning tree, note that 
for any vertex $w$ in $H$ there must exist a path in $G_1$ joining $w$ to $v$. 
In particular, there exists a path $P_w$ in $G_1\cap H$ from $w$ to either $v_1$ or $v_2$. 
Now $P_w\cup\{e\}$ contains a path in $H_1$ from $w$ to $v_1$. It follows that every vertex of $H$ is connected to every other vertex of $H$ in $H_1$. 
Thus $H_1$ is a connected spanning subgraph of $H$ with $|V(H)|-1$ edges and so $H_1$ is a tree. 

{\bf Case (2)}
Suppose $v_1v_2\in E(G)$. It may be assumed that two of the edges incident with $v$ belong to $G_1$. 
Thus $\S(v)$ is either the  centre or rightmost subgraph in Fig. \ref{Degree3VertexC}.
Let $H$ be the graph obtained by  adjoining the edges $e=v_1(sv_2)$ and $se$ to $G\backslash \{v,sv\}$.
Then $G$ is obtained from $H$ by applying a $(\bZ_2,\theta)$ $1$-extension  on the vertices $v_1$, $sv_1$ and $v_2$ and the edge $se$. 
Let $H_1=(G_1\cap H)\cup\{e,se\}$ and $H_2=G_2\cap H$. Then $H_1$ and $H_2$ are edge-disjoint $\bZ_2$-symmetric spanning subgraphs in $H$. Also, $H_2$ is clearly a tree. 
To show that $H_1$ is a tree, let $w$ be a vertex in $H$ and let $va$ and $vb$ be the two edges in $G_1$ which are incident with $v$. 
There must exist a path $P_{w}$ in $G_1\cap H$ from $w$ to either $a$ or $b$.
Also, there must exist a path $P$ in $G_1\cap H$ joining either $a$ or $b$ to either $sa$ or $sb$.
Now $P_w\cup P\cup \{e,se\}$ is a subgraph of $H_1$ which contains a path 
from $w$ to $v_1$. Thus $H_1$ is a connected spanning subgraph of $H$
with $|V(H)|-1$ edges and so $H_1$ is a spanning tree.

In both cases, no edge of $H$ is fixed by $s$ and so $(H,\theta)$ is an admissible pair. 
\endproof

\begin{remark}
The $(\bZ_2,\theta)$ $0$- and $1$-extensions also feature in the proof of a symmetric version of Laman's theorem for bar-joint frameworks in the Euclidean plane \cite{schulze,BS4}. They are analogous to the (non-symmetric) $0$- and $1$-extensions (or Henneberg moves) used in Laman's original theorem for generic frameworks \cite{Lamanbib}. 
\end{remark}

\subsection{Modified $(\bZ_2,\theta)$ $1$-extensions}
The third graph extension in the construction scheme is a variation of the $(\bZ_2,\theta)$ $1$-extension.

\begin{definition}
\label{CsXextension}
Let $G$ be a simple graph which is $\bZ_2$-symmetric with respect to $\theta$. 
Suppose there exists a graph $H$, distinct vertices $v_{1},v_{2},v_{3}\in V(H)$ and an edge $e\in E(H)$ with the following properties.
\begin{enumerate}[(i)]
\item $V(G)=V(H)\cup \{v,sv\}$ where $v,sv\notin V(H)$ and $v\not=sv$.
\item $H$ is $\bZ_2$-symmetric with respect to $\theta$, $v_1$, $v_2$, $sv_1$, and $sv_2$ are distinct and $e=v_1(sv_2)$.
\item $E(G)=\big(E(H)\setminus \big\{e,se\big\}\big)\cup \big\{vv_{i},s(vv_i)\,|\,\, i=1,2,3\big\}$.
\end{enumerate}
Then  $G$ is said to be obtained from $H$ by applying a {\em modified $(\bZ_2,\theta)$ $1$-extension} (on the vertices $v_1,v_2,v_3$ and the edge $e$). See also Fig.~\ref{XExt}.
\end{definition}

\begin{figure}[htp]
\begin{center}
\begin{tikzpicture}[very thick,scale=0.9]
\tikzstyle{every node}=[circle, draw=black, fill=white, inner sep=0pt, minimum width=5pt];
\filldraw[fill=black!03!white, draw=black, thin, dashed](0,0)circle(1.62cm);
\node[rectangle, draw=black!03!white,fill=black!03!white](l1) at (-0.5,1.03) {$v_1$};
\node[rectangle, draw=black!03!white,fill=black!03!white](l2) at (0.5,1.05) {$sv_1$};
\node[rectangle, draw=black!03!white,fill=black!03!white](l1) at (-0.8,-0.23) {$v_2$};
\node[rectangle, draw=black!03!white,fill=black!03!white](l2) at (0.8,-0.23) {$sv_2$};
\node[rectangle, draw=black!03!white,fill=black!03!white](l1) at (-0.38,-1.2) {$v_3$};
\node[rectangle, draw=black!03!white,fill=black!03!white](l2) at (0.32,-1.18) {$sv_3$};
\node (p1) at (-0.5,0.7) {};
\node (p2) at (0.5,0.7) {};
\node (p3) at (-0.8,0.13) {};
\node (p4) at (0.8,0.13) {};
\node (p5) at (-0.3,-0.83) {};
\node (p6) at (0.3,-0.83) {};
\draw(p1)--(p4);
\draw(p2)--(p3);
\filldraw[fill=black!50!white, draw=black, thick]
    (2.45,0) -- (3.05,0) -- (3.05,-0.1) -- (3.25,0.05) -- (3.05,0.2) -- (3.05,0.1) -- (2.45,0.1) -- (2.45,0);
\end{tikzpicture}
\hspace{0.5cm}
\begin{tikzpicture}[very thick,scale=0.9]
\tikzstyle{every node}=[circle, draw=black, fill=white, inner sep=0pt, minimum width=5pt];
\filldraw[fill=black!03!white, draw=black, thin, dashed](0,0)circle(1.62cm);
\node[rectangle, draw=black!03!white,fill=black!03!white](l1) at (-0.5,1.03) {$v_1$};
\node[rectangle, draw=black!03!white,fill=black!03!white](l2) at (0.5,1.05) {$sv_1$};
\node[rectangle, draw=black!03!white,fill=black!03!white](l1) at (-0.7,-0.23) {$v_2$};
\node[rectangle, draw=black!03!white,fill=black!03!white](l2) at (0.5,-0.23) {$sv_2$};
\node[rectangle, draw=black!03!white,fill=black!03!white](l1) at (-0.38,-1.2) {$v_3$};
\node[rectangle, draw=black!03!white,fill=black!03!white](l2) at (0.32,-1.18) {$sv_3$};
\node (p1) at (-0.5,0.7) {};
\node (p2) at (0.5,0.7) {};
\node (p3) at (-0.8,0.13) {};
\node (p4) at (0.8,0.13) {};
\node (p5) at (-0.3,-0.83) {};
\node (p6) at (0.3,-0.83) {};
\node (p7) at (-2,0) {};
\node (p8) at (2,0) {};
\node[rectangle, draw=white,fill=white](l1) at (-2.34,0) {$v$};
\node[rectangle, draw=white,fill=white](l2) at (2.38,0) {$sv$};
\draw(p7)--(p1);
\draw(p7)--(p3);
\draw(p7)--(p5);
\draw(p8)--(p2);
\draw(p8)--(p4);
\draw(p8)--(p6);
\end{tikzpicture} 
\end{center}
\vspace{-0.2cm}
\caption{\emph{A modified $(\bZ_2,\theta)$ $1$-extension $G$ of a graph $H$.}}
\label{XExt}
\end{figure}
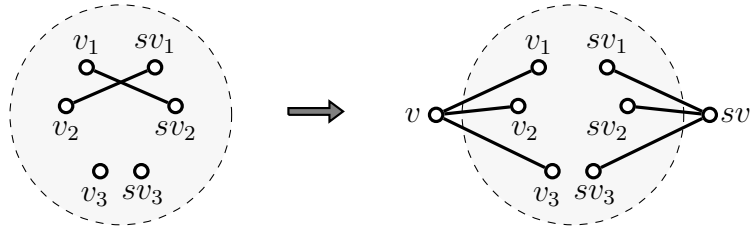

With the $(\bZ_2,\theta)$ $1$-extension and modified $(\bZ_2,\theta)$ $1$-extension it is possible to accommodate $3$-valent vertices $v$ for which $N(v)$ and $N(sv)$ are either disjoint or intersect in a single vertex. These two cases are considered separately.

\begin{lemma}
\label{0VertexLemma}
Let $(G,\theta)$ be an admissible pair and let $v$ be a $3$-valent vertex in $G$.
Suppose $N(v)$ and $N(sv)$ are disjoint. 
Then  there exists a graph $H$ with the following properties. 
\begin{enumerate}[(i)]
\item $G$ is obtained from $H$ by applying either a $(\bZ_2,\theta)$ $1$-extension or  a modified $(\bZ_2,\theta)$ $1$-extension.
\item $(H,\theta)$ is an admissible pair. 
\end{enumerate}
\end{lemma}

\proof
The graph $G$ may be expressed as a union of two edge-disjoint $\bZ_2$-symmetric spanning trees $G_1$ and $G_2$.
Let $N(v)=\{v_1,v_2,v_3\}$ and without loss of generality suppose that the edges $vv_1$ and $vv_2$ are both contained in the same spanning tree, $G_1$ say.
Then either $v_1v_2\notin E(G)$ or $v_1v_2\in E(G)$ and this edge belongs to $G_2$
(see Fig. \ref{Degree3Vertex}).
\begin{figure}[htp]
\begin{center}
\begin{tikzpicture}[very thick,scale=0.9]
\tikzstyle{every node}=[circle, draw=black, fill=white, inner sep=0pt, minimum width=5pt];
\filldraw[fill=black!03!white, draw=black, thin, dashed](0,0)circle(1.62cm);
\node[rectangle, draw=black!03!white,fill=black!03!white](l1) at (-0.5,1.03) {$v_1$};
\node[rectangle, draw=black!03!white,fill=black!03!white](l2) at (0.5,1.05) {$sv_1$};
\node[rectangle, draw=black!03!white,fill=black!03!white](l1) at (-0.7,-0.23) {$v_2$};
\node[rectangle, draw=black!03!white,fill=black!03!white](l2) at (0.5,-0.23) {$sv_2$};
\node[rectangle, draw=black!03!white,fill=black!03!white](l1) at (-0.38,-1.2) {$v_3$};
\node[rectangle, draw=black!03!white,fill=black!03!white](l2) at (0.32,-1.2) {$sv_3$};
\node (p1) at (-0.4,0.75) {};
\node (p2) at (0.4,0.75) {};
\node (p3) at (-0.8,0.1) {};
\node (p4) at (0.8,0.1) {};
\node (p5) at (-0.3,-0.83) {};
\node (p6) at (0.3,-0.83) {};
\node (p7) at (-2,0) {};
\node (p8) at (2,0) {};
\node[rectangle, draw=white,fill=white,below right](l1) at (-2.25,-0.2) {$v$};
\node[rectangle, draw=white,fill=white,below left](l2) at (2.3,-0.2) {$sv$};
\draw(p7)--(p1);
\draw(p7)--(p3);
\draw[lightgray](p7)--(p5);
\draw(p8)--(p2);
\draw(p8)--(p4);
\draw[lightgray](p8)--(p6);
\end{tikzpicture}
\hspace{0.8cm}
\begin{tikzpicture}[very thick,scale=0.9]
\tikzstyle{every node}=[circle, draw=black, fill=white, inner sep=0pt, minimum width=5pt];
\filldraw[fill=black!03!white, draw=black, thin, dashed](0,0)circle(1.62cm);
\node[rectangle, draw=black!03!white,fill=black!03!white](l1) at (-0.5,1.03) {$v_1$};
\node[rectangle, draw=black!03!white,fill=black!03!white](l2) at (0.5,1.05) {$sv_1$};
\node[rectangle, draw=black!03!white,fill=black!03!white](l1) at (-0.7,-0.23) {$v_2$};
\node[rectangle, draw=black!03!white,fill=black!03!white](l2) at (0.5,-0.23) {$sv_2$};
\node[rectangle, draw=black!03!white,fill=black!03!white](l1) at (-0.38,-1.2) {$v_3$};
\node[rectangle, draw=black!03!white,fill=black!03!white](l2) at (0.32,-1.2) {$sv_3$};
\node (p1) at (-0.4,0.75) {};
\node (p2) at (0.4,0.75) {};
\node (p3) at (-0.8,0.1) {};
\node (p4) at (0.8,0.1) {};
\node (p5) at (-0.3,-0.83) {};
\node (p6) at (0.3,-0.83) {};
\node (p7) at (-2,0) {};
\node (p8) at (2,0) {};
\node[rectangle, draw=white,fill=white,below right](l1) at (-2.25,-0.2) {$v$};
\node[rectangle, draw=white,fill=white,below left](l2) at (2.3,-0.2) {$sv$};
\draw(p7)--(p1);
\draw(p7)--(p3);
\draw[lightgray](p7)--(p5);
\draw(p8)--(p2);
\draw(p8)--(p4);
\draw[lightgray](p8)--(p6);
\draw[lightgray](p1)--(p3);
\draw[lightgray](p2)--(p4);
\end{tikzpicture}
\end{center}
\vspace{-0.2cm}
\caption{$N(v)$ and $N(sv)$ are disjoint and $vv_1$ and $vv_2$ belong to the same spanning tree.}
\label{Degree3Vertex}
\end{figure}
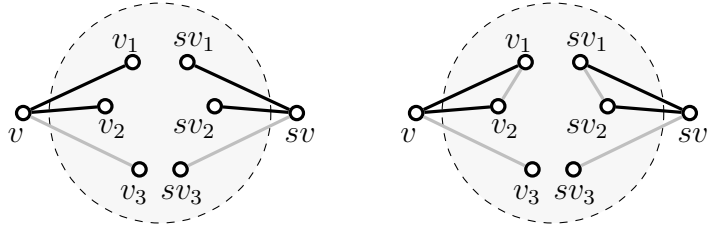
If $v_1v_2\notin E(G)$ then let $e=v_1v_2$.
If $v_1v_2\in E(G)$ then let $e=v_1(sv_2)$ and 
note that $G$ does not contain the edges $e$ and $se$ since, otherwise, 
either $G_1$ or $G_2$ would contain a cycle.
Let $H$ be the graph obtained by adjoining the edges $e$ and $se$ to $G\backslash\{v,sv\}$.
If $v_1v_2\notin E(G)$ then  $G$ is obtained from $H$ by applying a $(\bZ_2,\theta)$ $1$-extension on the vertices $v_1$, $v_2$, $v_3$ and  the edge $e$. 
If $v_1v_2\in E(G)$ then  $G$ is obtained from $H$ by applying a modified $(\bZ_2,\theta)$ $1$-extension on the vertices $v_1$, $v_2$, $v_3$ and  the edge $e$. 
In both cases, the arguments of Lemma \ref{2VertexLemma} may be applied to show that $H_1=(G_1\cap H)\cup\{e,se\}$ and $H_2=G_2\cap H$ are edge-disjoint $\bZ_2$-symmetric spanning trees in $H$. 
It follows that $(H,\theta)$ is an admissible pair.
\endproof

In proving the preceding lemmas, the spanning trees $H_1$ and $H_2$ were constructed directly from the subgraphs $G_1\cap H$ and $G_2\cap H$.
If $N(v)$ and $N(sv)$ intersect in one vertex then this direct construction is not always possible, as shown in the proof of the following lemma.  

\begin{lemma}
\label{1VertexLemma}
Let $(G,\theta)$ be an admissible pair and let $v$ be a $3$-valent vertex in $G$.
Suppose $N(v)$ and $N(sv)$ intersect in one vertex. 
Then  there exists a graph $H$ with the following properties. 
\begin{enumerate}[(i)]
\item $G$ may be obtained from $H$ by applying either a $(\bZ_2,\theta)$ $1$-extension or  a modified $(\bZ_2,\theta)$ $1$-extension.
\item $(H,\theta)$ is an admissible pair. 
\end{enumerate}
\end{lemma}

\proof
Let $N(v)=\{v_1,v_2,v_3\}$ and note that, by Lemma \ref{NbdLemma}, $N(v)$ and $N(sv)$ must intersect in the unique vertex $v_0$ which is fixed by $s$.
Thus $v_i=v_0$ for some $i\in\{1,2,3\}$.
The graph $G$ may be expressed as a union of two edge-disjoint $\bZ_2$-symmetric spanning trees $G_1$ and $G_2$ and without loss of generality it may be assumed that the edges $vv_1$ and $vv_2$ are both edges of $G_1$.  
If $v_1v_2\notin E(G)$, or, $v_1v_2\in E(G)$ and $v_0=v_3$  then 
the proof proceeds by adapting the arguments of Lemma \ref{0VertexLemma}.
It remains to consider the case where $v_1v_2\in E(G)$ and $v_0\not= v_3$.
Note that $N(v)=\{v_0,v_j,v_3\}$ where either $v_0=v_1$ and $v_j=v_2$, or, $v_0=v_2$ and $v_j=v_1$. Thus $\S(v)$ is one of the graphs in Fig. \ref{Degree3VertexB}.
Since $G_1$ is a spanning tree in $G$ there must exist a path $P$ in $G_1\backslash\{v,sv\}$ from $v_3$ to exactly one of the vertices $v_0$, $v_j$ or $sv_j$. These three possible cases are considered separately.

{\bf Case (1)}
Suppose $P$ joins $v_3$ to the fixed vertex $v_0$.
Then $G_1\backslash\{v,sv\}$ does not contain paths which join either $v_3$ to $v_j$, or, $v_3$ to $sv_j$ since such a path would imply the existence of a cycle in $G_1$.
In particular, if $e=v_jv_3$ or $f=(sv_j)v_3$ is an edge of $G$ then this edge must belong to $G_2$.
It follows that $e$ and  $f$ are not both edges of $G$ since, if they were, then $G_2$ would contain a cycle with edges $e$, $f$, $se$ and $sf$.
From these observations, it follows that $\S(v)$ must appear in the first three rows of Fig. \ref{Degree3VertexB}.
If $e\notin E(G)$ then let $H$ be the graph obtained by adjoining the edges $e$ and $se$  to $G\backslash\{v,sv\}$ and
let $H_1=(G_1\cap H)\cup\{e,se\}$ and $H_2=G_2\cap H$. 
If $e\in E(G)$ then $f\notin E(G)$ and $H$ is constructed by adjoining the edges $f$ and $sf$ to $G\backslash\{v,sv\}$.
Let $H_1=(G_1\cap H)\cup\{f,sf\}$ and $H_2=G_2\cap H$.

{\bf Case (2)}
Suppose $P$ joins $v_3$ to $v_j$.
Then $G_1\backslash\{v,sv\}$ does not contain paths which join either $v_3$ to $v_0$, or, $v_3$ to $sv_j$.
In particular, either $e=v_0v_3\notin E(G)$ or $f=(sv_j)v_3\notin E(G)$, since otherwise both edges belong to $G_2$ and this creates a cycle in $G_2$.
Thus $\S(v)$ appears in the first, second or fourth row of Fig. \ref{Degree3VertexB}.
If $e\notin E(G)$ then  let $H$ be the graph obtained by adjoining the edges $e$ and $se$  to $G\backslash\{v,sv\}$ and 
let $H_1=(G_1\cap H)\cup\{e,se\}$ and $H_2=G_2\cap H$. 
If $e\in E(G)$ then $f\notin E(G)$ and $e$ is contained in the spanning tree $G_2$.
Let $H$ be the graph  obtained by adjoining the edges $f$ and $sf$  to $G\backslash\{v,sv\}$.
Note that the subgraph $(G_1\cap H)\cup\{f,sf\}$ must contain a cycle and so, instead, adjoin the edges $f$ and $sf$ to $G_2\cap H$ and re-colour the edges 
$g=v_0v_j$ and $sg$, both of which are contained in $G_2$. 
Thus, let $H_1=(G_1\cap H)\cup\{g,sg\}$ and $H_2=((G_2\cap H)\backslash\{g,sg\})\cup\{f,sf\}$.

\begin{figure}[htp]
\begin{tabular}{c}
\begin{tikzpicture}[very thick,scale=1]
\tikzstyle{every node}=[circle, draw=black, fill=white, inner sep=0pt, minimum width=5pt];
\filldraw[fill=black!03!white, draw=black, thin, dashed](0,0)circle(1.62cm);
\node[rectangle, draw=black!03!white,fill=black!03!white](l1) at (0,1.1) {$v_1=v_0$};

\node[rectangle, draw=black!03!white,fill=black!03!white](l1) at (-0.7,-0.23) {$v_2$};
\node[rectangle, draw=black!03!white,fill=black!03!white](l2) at (0.5,-0.23) {$sv_2$};
\node[rectangle, draw=black!03!white,fill=black!03!white](l1) at (-0.38,-1.2) {$v_3$};
\node[rectangle, draw=black!03!white,fill=black!03!white](l2) at (0.32,-1.2) {$sv_3$};
\node (p1) at (0,0.85) {};
\node (p3) at (-0.8,0.1) {};
\node (p4) at (0.8,0.1) {};
\node (p5) at (-0.3,-0.83) {};
\node (p6) at (0.3,-0.83) {};
\node (p7) at (-2,0) {};
\node (p8) at (2,0) {};
\node[rectangle, draw=white,fill=white,below right](l1) at (-2.25,-0.2) {$v$};
\node[rectangle, draw=white,fill=white,below left](l2) at (2.3,-0.2) {$sv$};
\draw(p7)--(p1);
\draw(p7)--(p3);
\draw[lightgray](p7)--(p5);
\draw(p8)--(p1);
\draw(p8)--(p4);
\draw[lightgray](p8)--(p6);
\draw[lightgray](p1)--(p3);
\draw[lightgray](p1)--(p4);
\end{tikzpicture}

\hspace{0.7cm}
\begin{tikzpicture}[very thick,scale=1]
\tikzstyle{every node}=[circle, draw=black, fill=white, inner sep=0pt, minimum width=5pt];
\filldraw[fill=black!03!white, draw=black, thin, dashed](0,0)circle(1.62cm);
\node[rectangle, draw=black!03!white,fill=black!03!white](l1) at (-0.5,1.03) {$v_1$};
\node[rectangle, draw=black!03!white,fill=black!03!white](l2) at (0.5,1.05) {$sv_1$};
\node[rectangle, draw=black!03!white,fill=black!03!white](l2) at (0.44,-0.35) {$v_2\, = v_0$};
\node[rectangle, draw=black!03!white,fill=black!03!white](l1) at (-0.38,-1.2) {$v_3$};
\node[rectangle, draw=black!03!white,fill=black!03!white](l2) at (0.5,-1.2) {$sv_3$};
\node (p1) at (-0.7,0.75) {};
\node (p2) at (0.7,0.75) {};
\node (p3) at (0,0.1) {};
\node (p5) at (-0.6,-0.83) {};
\node (p6) at (0.6,-0.83) {};
\node (p7) at (-2,0) {};
\node (p8) at (2,0) {};
\node[rectangle, draw=white,fill=white,below right](l1) at (-2.25,-0.2) {$v$};
\node[rectangle, draw=white,fill=white,below left](l2) at (2.3,-0.2) {$sv$};
\draw(p7)--(p1);
\draw(p7)--(p3);
\draw[lightgray](p7)--(p5);
\draw(p8)--(p2);
\draw(p8)--(p3);
\draw[lightgray](p8)--(p6);
\draw[lightgray](p5)--(p3);
\draw[lightgray](p6)--(p3);
\draw[lightgray](p1)--(p3);
\draw[lightgray](p2)--(p3);
\end{tikzpicture} 
\\~\\
\begin{tikzpicture}[very thick,scale=1]
\tikzstyle{every node}=[circle, draw=black, fill=white, inner sep=0pt, minimum width=5pt];
\filldraw[fill=black!03!white, draw=black, thin, dashed](0,0)circle(1.62cm);
\node[rectangle, draw=black!03!white,fill=black!03!white](l1) at (0,1.1) {$v_1=v_0$};
\node[rectangle, draw=black!03!white,fill=black!03!white,right](l1) at (-0.65,-0.05) {$v_2$};
\node[rectangle, draw=black!03!white,fill=black!03!white,left](l2) at (0.6,-0.05) {$sv_2$};
\node[rectangle, draw=black!03!white,fill=black!03!white](l1) at (-0.38,-1.2) {$v_3$};
\node[rectangle, draw=black!03!white,fill=black!03!white](l2) at (0.32,-1.2) {$sv_3$};
\node (p1) at (0,0.85) {};
\node (p3) at (-0.8,0.1) {};
\node (p4) at (0.8,0.1) {};
\node (p5) at (-0.3,-0.83) {};
\node (p6) at (0.3,-0.83) {};
\node (p7) at (-2,0) {};
\node (p8) at (2,0) {};
\node[rectangle, draw=white,fill=white,below right](l1) at (-2.25,-0.2) {$v$};
\node[rectangle, draw=white,fill=white,below left](l2) at (2.3,-0.2) {$sv$};
\draw(p7)--(p1);
\draw(p7)--(p3);
\draw[lightgray](p7)--(p5);
\draw(p8)--(p1);
\draw(p8)--(p4);
\draw[lightgray](p8)--(p6);
\draw[lightgray](p5)--(p3);
\draw[lightgray](p6)--(p4);
\draw[lightgray](p1)--(p3);
\draw[lightgray](p1)--(p4);
\end{tikzpicture} 
\hspace{0.7cm}
\begin{tikzpicture}[very thick,scale=1]
\tikzstyle{every node}=[circle, draw=black, fill=white, inner sep=0pt, minimum width=5pt];
\filldraw[fill=black!03!white, draw=black, thin, dashed](0,0)circle(1.62cm);
\node[rectangle, draw=black!03!white,fill=black!03!white](l1) at (0,1.1) {$v_1=v_0$};
\node[rectangle, draw=black!03!white,fill=black!03!white,left](l1) at (-0.7,-0.23) {$v_2$};
\node[rectangle, draw=black!03!white,fill=black!03!white,right](l2) at (0.7,-0.23) {$sv_2$};
\node[rectangle, draw=black!03!white,fill=black!03!white](l1) at (-0.38,-1.2) {$v_3$};
\node[rectangle, draw=black!03!white,fill=black!03!white](l2) at (0.32,-1.2) {$sv_3$};
\node (p1) at (0,0.85) {};
\node (p3) at (-0.8,0.1) {};
\node (p4) at (0.8,0.1) {};
\node (p5) at (-0.3,-0.83) {};
\node (p6) at (0.3,-0.83) {};
\node (p7) at (-2,0) {};
\node (p8) at (2,0) {};
\node[rectangle, draw=white,fill=white,below right](l1) at (-2.25,-0.2) {$v$};
\node[rectangle, draw=white,fill=white,below left](l2) at (2.3,-0.2) {$sv$};
\draw(p7)--(p1);
\draw(p7)--(p3);
\draw[lightgray](p7)--(p5);
\draw(p8)--(p1);
\draw(p8)--(p4);
\draw[lightgray](p8)--(p6);
\draw[lightgray](p3)--(p6);
\draw[lightgray](p4)--(p5);
\draw[lightgray](p1)--(p3);
\draw[lightgray](p1)--(p4);
\end{tikzpicture}
\\~\\
\begin{tikzpicture}[very thick,scale=1]
\tikzstyle{every node}=[circle, draw=black, fill=white, inner sep=0pt, minimum width=5pt];
\filldraw[fill=black!03!white, draw=black, thin, dashed](0,0)circle(1.62cm);
\node[rectangle, draw=black!03!white,fill=black!03!white](l1) at (-0.5,1.03) {$v_1$};
\node[rectangle, draw=black!03!white,fill=black!03!white](l2) at (0.5,1.05) {$sv_1$};
\node[rectangle, draw=black!03!white,fill=black!03!white](l2) at (0.44,-0.35) {$v_2\, = v_0$};
\node[rectangle, draw=black!03!white,fill=black!03!white](l1) at (-0.38,-1.2) {$v_3$};
\node[rectangle, draw=black!03!white,fill=black!03!white](l2) at (0.5,-1.2) {$sv_3$};
\node (p1) at (-0.7,0.75) {};
\node (p2) at (0.7,0.75) {};
\node (p3) at (0,0.1) {};
\node (p5) at (-0.6,-0.83) {};
\node (p6) at (0.6,-0.83) {};
\node (p7) at (-2,0) {};
\node (p8) at (2,0) {};
\node[rectangle, draw=white,fill=white,below right](l1) at (-2.25,-0.2) {$v$};
\node[rectangle, draw=white,fill=white,below left](l2) at (2.3,-0.2) {$sv$};
\draw(p7)--(p1);
\draw(p7)--(p3);
\draw[lightgray](p7)--(p5);
\draw(p8)--(p2);
\draw(p8)--(p3);
\draw[lightgray](p8)--(p6);
\draw(p5)--(p3);
\draw(p6)--(p3);
\draw[lightgray](p1)--(p3);
\draw[lightgray](p2)--(p3);
\end{tikzpicture}
\hspace{0.7cm}
\begin{tikzpicture}[very thick,scale=1]
\tikzstyle{every node}=[circle, draw=black, fill=white, inner sep=0pt, minimum width=5pt];
\filldraw[fill=black!03!white, draw=black, thin, dashed](0,0)circle(1.62cm);
\node[rectangle, draw=black!03!white,fill=black!03!white](l1) at (0,1.1) {$v_1=v_0$};
\node[rectangle, draw=black!03!white,fill=black!03!white,left](l1) at (-0.7,-0.23) {$v_2$};
\node[rectangle, draw=black!03!white,fill=black!03!white,right](l2) at (0.7,-0.23) {$sv_2$};
\node[rectangle, draw=black!03!white,fill=black!03!white](l1) at (-0.38,-1.2) {$v_3$};
\node[rectangle, draw=black!03!white,fill=black!03!white](l2) at (0.32,-1.2) {$sv_3$};
\node (p1) at (0,0.85) {};
\node (p3) at (-0.8,0.1) {};
\node (p4) at (0.8,0.1) {};
\node (p5) at (-0.3,-0.83) {};
\node (p6) at (0.3,-0.83) {};
\node (p7) at (-2,0) {};
\node (p8) at (2,0) {};
\node[rectangle, draw=white,fill=white,below right](l1) at (-2.25,-0.2) {$v$};
\node[rectangle, draw=white,fill=white,below left](l2) at (2.3,-0.2) {$sv$};
\draw(p7)--(p1);
\draw(p7)--(p3);
\draw[lightgray](p7)--(p5);
\draw(p8)--(p1);
\draw(p8)--(p4);
\draw[lightgray](p8)--(p6);
\draw[lightgray](p3)--(p5);
\draw[lightgray](p4)--(p6);
\draw(p1)--(p6);
\draw(p1)--(p5);
\draw[lightgray](p1)--(p3);
\draw[lightgray](p1)--(p4);
\end{tikzpicture}
\hspace{0.7cm}
 \begin{tikzpicture}[very thick,scale=1]
\tikzstyle{every node}=[circle, draw=black, fill=white, inner sep=0pt, minimum width=5pt];
\filldraw[fill=black!03!white, draw=black, thin, dashed](0,0)circle(1.62cm);
\node[rectangle, draw=black!03!white,fill=black!03!white](l1) at (0,1.1) {$v_1=v_0$};
\node[rectangle, draw=black!03!white,fill=black!03!white,left](l1) at (-0.7,-0.23) {$v_2$};
\node[rectangle, draw=black!03!white,fill=black!03!white,right](l2) at (0.7,-0.23) {$sv_2$};
\node[rectangle, draw=black!03!white,fill=black!03!white](l1) at (-0.38,-1.2) {$v_3$};
\node[rectangle, draw=black!03!white,fill=black!03!white](l2) at (0.32,-1.2) {$sv_3$};
\node (p1) at (0,0.85) {};
\node (p3) at (-0.8,0.1) {};
\node (p4) at (0.8,0.1) {};
\node (p5) at (-0.3,-0.83) {};
\node (p6) at (0.3,-0.83) {};
\node (p7) at (-2,0) {};
\node (p8) at (2,0) {};
\node[rectangle, draw=white,fill=white,below right](l1) at (-2.25,-0.2) {$v$};
\node[rectangle, draw=white,fill=white,below left](l2) at (2.3,-0.2) {$sv$};
\draw(p7)--(p1);
\draw(p7)--(p3);
\draw[lightgray](p7)--(p5);
\draw(p8)--(p1);
\draw(p8)--(p4);
\draw[lightgray](p8)--(p6);
\draw[lightgray](p3)--(p6);
\draw[lightgray](p4)--(p5);
\draw[lightgray](p1)--(p3);
\draw[lightgray](p1)--(p4);
\draw(p5)--(p1);
\draw(p6)--(p1);
\end{tikzpicture}
\\~\\
\begin{tikzpicture}[very thick,scale=1]
\tikzstyle{every node}=[circle, draw=black, fill=white, inner sep=0pt, minimum width=5pt];
\filldraw[fill=black!03!white, draw=black, thin, dashed](0,0)circle(1.62cm);
\node[rectangle, draw=black!03!white,fill=black!03!white](l1) at (0,1.1) {$v_1=v_0$};
\node[rectangle, draw=black!03!white,fill=black!03!white,right](l1) at (-0.65,-0.05) {$v_2$};
\node[rectangle, draw=black!03!white,fill=black!03!white,left](l2) at (0.6,-0.05) {$sv_2$};
\node[rectangle, draw=black!03!white,fill=black!03!white](l1) at (-0.38,-1.2) {$v_3$};
\node[rectangle, draw=black!03!white,fill=black!03!white](l2) at (0.32,-1.2) {$sv_3$};
\node (p1) at (0,0.85) {};
\node (p3) at (-0.8,0.1) {};
\node (p4) at (0.8,0.1) {};
\node (p5) at (-0.3,-0.83) {};
\node (p6) at (0.3,-0.83) {};
\node (p7) at (-2,0) {};
\node (p8) at (2,0) {};
\node[rectangle, draw=white,fill=white,below right](l1) at (-2.25,-0.2) {$v$};
\node[rectangle, draw=white,fill=white,below left](l2) at (2.3,-0.2) {$sv$};
\draw(p7)--(p1);
\draw(p7)--(p3);
\draw[lightgray](p7)--(p5);
\draw(p8)--(p1);
\draw(p8)--(p4);
\draw[lightgray](p8)--(p6);
\draw(p5)--(p3);
\draw(p6)--(p4);
\draw[lightgray](p1)--(p3);
\draw[lightgray](p1)--(p4);
\end{tikzpicture}
\hspace{0.7cm}
\begin{tikzpicture}[very thick,scale=1]
\tikzstyle{every node}=[circle, draw=black, fill=white, inner sep=0pt, minimum width=5pt];
\filldraw[fill=black!03!white, draw=black, thin, dashed](0,0)circle(1.62cm);
\node[rectangle, draw=black!03!white,fill=black!03!white](l1) at (0,1.13) {$v_1=v_0$};
\node[rectangle, draw=black!03!white,fill=black!03!white,left](l1) at (-0.7,-0.23) {$v_2$};
\node[rectangle, draw=black!03!white,fill=black!03!white,right](l2) at (0.7,-0.23) {$sv_2$};
\node[rectangle, draw=black!03!white,fill=black!03!white](l1) at (-0.38,-1.2) {$v_3$};
\node[rectangle, draw=black!03!white,fill=black!03!white](l2) at (0.32,-1.2) {$sv_3$};
\node (p1) at (0,0.85) {};
\node (p3) at (-0.8,0.1) {};
\node (p4) at (0.8,0.1) {};
\node (p5) at (-0.3,-0.83) {};
\node (p6) at (0.3,-0.83) {};
\node (p7) at (-2,0) {};
\node (p8) at (2,0) {};
\node[rectangle, draw=white,fill=white,below right](l1) at (-2.25,-0.2) {$v$};
\node[rectangle, draw=white,fill=white,below left](l2) at (2.3,-0.2) {$sv$};
\draw(p7)--(p1);
\draw(p7)--(p3);
\draw[lightgray](p7)--(p5);
\draw(p8)--(p1);
\draw(p8)--(p4);
\draw[lightgray](p8)--(p6);
\draw[lightgray](p1)--(p4);
\draw[lightgray](p1)--(p3);
\draw(p3)--(p5);
\draw(p4)--(p6);
\draw[lightgray](p3)--(p6);
\draw[lightgray](p4)--(p5);
\end{tikzpicture}

\hspace{0.7cm}
\begin{tikzpicture}[very thick,scale=1]
\tikzstyle{every node}=[circle, draw=black, fill=white, inner sep=0pt, minimum width=5pt];
\filldraw[fill=black!03!white, draw=black, thin, dashed](0,0)circle(1.62cm);
\node[rectangle, draw=black!03!white,fill=black!03!white](l1) at (0,1.1) {$v_1=v_0$};
\node[rectangle, draw=black!03!white,fill=black!03!white,below left](l1) at (-0.68,-0.1) {$v_2$};
\node[rectangle, draw=black!03!white,fill=black!03!white,below right](l2) at (0.68,-0.1) {$sv_2$};
\node[rectangle, draw=black!03!white,fill=black!03!white](l1) at (-0.38,-1.2) {$v_3$};
\node[rectangle, draw=black!03!white,fill=black!03!white](l2) at (0.32,-1.2) {$sv_3$};
\node (p1) at (0,0.85) {};
\node (p3) at (-0.8,0.1) {};
\node (p4) at (0.8,0.1) {};
\node (p5) at (-0.3,-0.83) {};
\node (p6) at (0.3,-0.83) {};
\node (p7) at (-2,0) {};
\node (p8) at (2,0) {};
\node[rectangle, draw=white,fill=white,below right](l1) at (-2.25,-0.2) {$v$};
\node[rectangle, draw=white,fill=white,below left](l2) at (2.3,-0.2) {$sv$};
\draw(p7)--(p1);
\draw(p7)--(p3);
\draw[lightgray](p7)--(p5);
\draw(p8)--(p1);
\draw(p8)--(p4);
\draw[lightgray](p8)--(p6);
\draw(p5)--(p3);
\draw(p6)--(p4);
\draw[lightgray](p1)--(p3);
\draw[lightgray](p1)--(p4);
\draw[lightgray](p1)--(p5);
\draw[lightgray](p1)--(p6);
\end{tikzpicture}
\\~\\
\begin{tikzpicture}[very thick,scale=1]
\tikzstyle{every node}=[circle, draw=black, fill=white, inner sep=0pt, minimum width=5pt];
\filldraw[fill=black!03!white, draw=black, thin, dashed](0,0)circle(1.62cm);
\node[rectangle, draw=black!03!white,fill=black!03!white](l1) at (0,1.1) {$v_1=v_0$};
\node[rectangle, draw=black!03!white,fill=black!03!white,left](l1) at (-0.7,-0.23) {$v_2$};
\node[rectangle, draw=black!03!white,fill=black!03!white,right](l2) at (0.7,-0.23) {$sv_2$};
\node[rectangle, draw=black!03!white,fill=black!03!white](l1) at (-0.38,-1.2) {$v_3$};
\node[rectangle, draw=black!03!white,fill=black!03!white](l2) at (0.32,-1.2) {$sv_3$};
\node (p1) at (0,0.85) {};
\node (p3) at (-0.8,0.1) {};
\node (p4) at (0.8,0.1) {};
\node (p5) at (-0.3,-0.83) {};
\node (p6) at (0.3,-0.83) {};
\node (p7) at (-2,0) {};
\node (p8) at (2,0) {};
\node[rectangle, draw=white,fill=white,below right](l1) at (-2.25,-0.2) {$v$};
\node[rectangle, draw=white,fill=white,below left](l2) at (2.3,-0.2) {$sv$};
\draw(p7)--(p1);
\draw(p7)--(p3);
\draw[lightgray](p7)--(p5);
\draw(p8)--(p1);
\draw(p8)--(p4);
\draw[lightgray](p8)--(p6);
\draw(p3)--(p6);
\draw(p4)--(p5);
\draw[lightgray](p1)--(p3);
\draw[lightgray](p1)--(p4);
\end{tikzpicture}
\hspace{0.7cm}
\begin{tikzpicture}[very thick,scale=1]
\tikzstyle{every node}=[circle, draw=black, fill=white, inner sep=0pt, minimum width=5pt];
\filldraw[fill=black!03!white, draw=black, thin, dashed](0,0)circle(1.62cm);
\node[rectangle, draw=black!03!white,fill=black!03!white](l1) at (0,1.13) {$v_1=v_0$};
\node[rectangle, draw=black!03!white,fill=black!03!white,left](l1) at (-0.7,-0.23) {$v_2$};
\node[rectangle, draw=black!03!white,fill=black!03!white,right](l2) at (0.7,-0.23) {$sv_2$};
\node[rectangle, draw=black!03!white,fill=black!03!white](l1) at (-0.38,-1.2) {$v_3$};
\node[rectangle, draw=black!03!white,fill=black!03!white](l2) at (0.32,-1.2) {$sv_3$};
\node (p1) at (0,0.85) {};
\node (p3) at (-0.8,0.1) {};
\node (p4) at (0.8,0.1) {};
\node (p5) at (-0.3,-0.83) {};
\node (p6) at (0.3,-0.83) {};
\node (p7) at (-2,0) {};
\node (p8) at (2,0) {};
\node[rectangle, draw=white,fill=white,below right](l1) at (-2.25,-0.2) {$v$};
\node[rectangle, draw=white,fill=white,below left](l2) at (2.3,-0.2) {$sv$};
\draw(p7)--(p1);
\draw(p7)--(p3);
\draw[lightgray](p7)--(p5);
\draw(p8)--(p1);
\draw(p8)--(p4);
\draw[lightgray](p8)--(p6);
\draw[lightgray](p1)--(p4);
\draw[lightgray](p1)--(p3);
\draw[lightgray](p3)--(p5);
\draw[lightgray](p4)--(p6);
\draw(p3)--(p6);
\draw(p4)--(p5);
\end{tikzpicture}
\hspace{0.7cm}
\begin{tikzpicture}[very thick,scale=1]
\tikzstyle{every node}=[circle, draw=black, fill=white, inner sep=0pt, minimum width=5pt];
\filldraw[fill=black!03!white, draw=black, thin, dashed](0,0)circle(1.62cm);
\node[rectangle, draw=black!03!white,fill=black!03!white](l1) at (-0.84,1.0) {$v_1$};
\node[rectangle, draw=black!03!white,fill=black!03!white](l2) at (0.84,1.0) {$sv_1$};
\node[rectangle, draw=black!03!white,fill=black!03!white,above](l1) at (0,0.64) {$v_2=v_0$};
\node[rectangle, draw=black!03!white,fill=black!03!white](l1) at (-0.38,-1.2) {$v_3$};
\node[rectangle, draw=black!03!white,fill=black!03!white](l2) at (0.32,-1.2) {$sv_3$};
\node (p1) at (-0.8,0.75) {};
\node (p2) at (0.8,0.75) {};
\node (p3) at (0,0.5) {};
\node (p5) at (-0.3,-0.83) {};
\node (p6) at (0.3,-0.83) {};
\node (p7) at (-2,0) {};
\node (p8) at (2,0) {};
\node[rectangle, draw=white,fill=white,below right](l1) at (-2.25,-0.2) {$v$};
\node[rectangle, draw=white,fill=white,below left](l2) at (2.3,-0.2) {$sv$};
\draw(p7)--(p1);
\draw(p7)--(p3);
\draw[lightgray](p7)--(p5);
\draw(p8)--(p2);
\draw(p8)--(p3);
\draw[lightgray](p8)--(p6);
\draw[lightgray](p1)--(p3);
\draw[lightgray](p2)--(p3);
\draw[lightgray](p3)--(p5);
\draw[lightgray](p3)--(p6);
\draw(p1)--(p6);
\draw(p2)--(p5);

\end{tikzpicture} 

\end{tabular}
\vspace{0.3cm}
\caption{$N(v)$ and $N(sv)$ intersect in one vertex, 
$vv_1$ and $vv_2$ belong to the same spanning tree, $v_0\not=v_3$ and $v_1v_2\in E(G)$.}
\label{Degree3VertexB}
\end{figure}
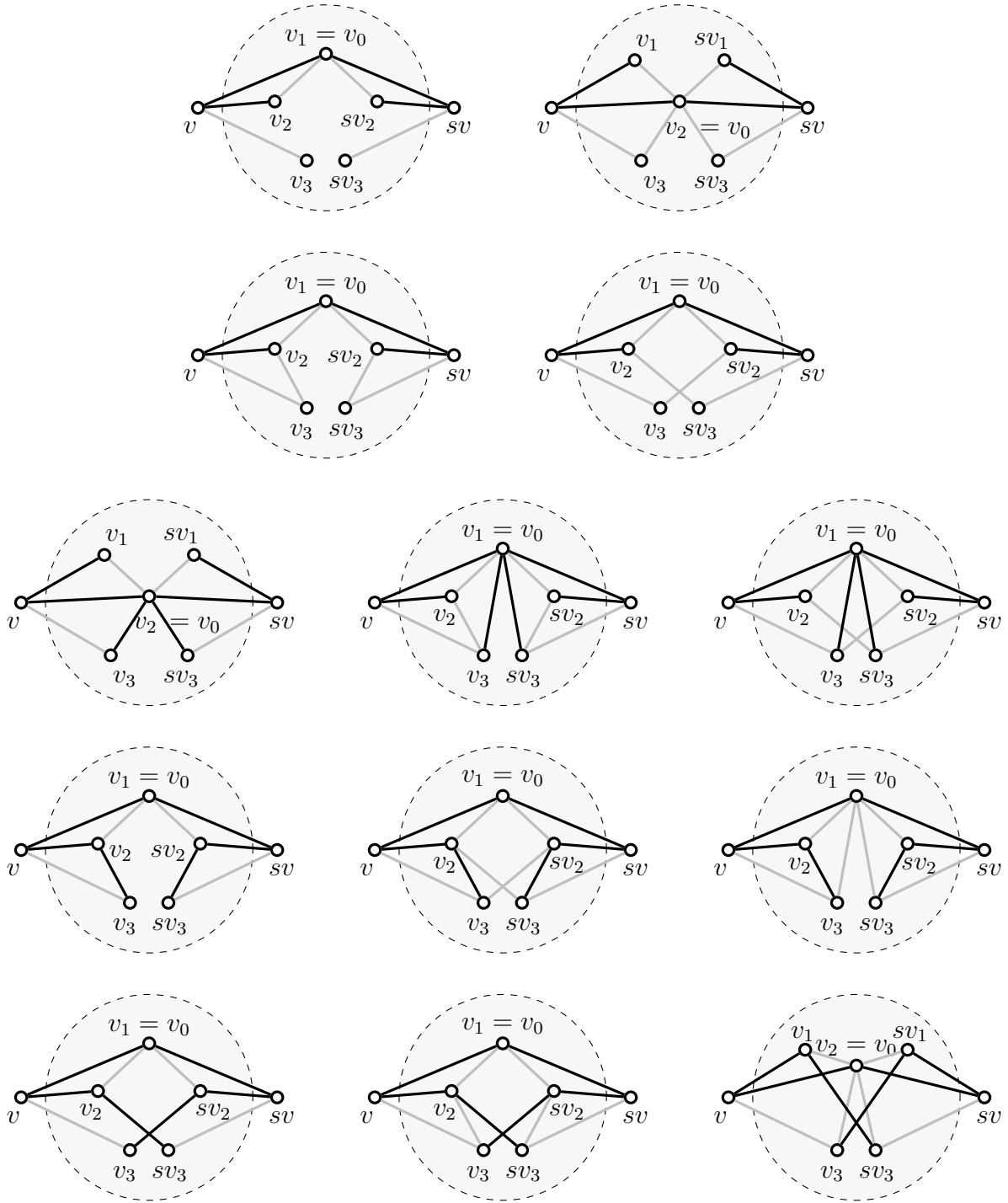

{\bf Case (3)}
Suppose $P$ joins $v_3$ to $sv_j$.
Then $G_1\backslash\{v,sv\}$ does not contain paths which join either $v_3$ to $v_0$, or, $v_3$ to $v_j$.
In particular, either $e=v_0v_3\notin E(G)$ or $f=v_jv_3\notin E(G)$.
Thus $\S(v)$ appears among the first, second and fifth rows of Fig. \ref{Degree3VertexB}.
If $e\notin E(G)$ then let $H$ be the graph obtained by  adjoining the edges $e$ and $se$ to $G\backslash\{v,sv\}$ and 
let $H_1=(G_1\cap H)\cup\{e,se\}$ and $H_2=G_2\cap H$.
If $e\in E(G)$ then $f\notin E(G)$ and $H$ is obtained by  adjoining the edges $f$ and $sf$  to $G\backslash\{v,sv\}$.
Let $H_1=(G_1\cap H)\cup\{g, sg\}$ and $H_2=((G_2\cap H)\backslash\{g,sg\})\cup\{f,sf\}$ where $g=v_0v_j$. 

In each of the above cases $H$ is an edge-disjoint union of the $\bZ_2$-symmetric spanning trees $H_1$ and $H_2$. Moreover, $H$ has no fixed edges and so $(H,\theta)$ is an admissible pair. 
If  $e\notin E(G)$ then $G$ is obtained from $H$ by applying a $(\bZ_2,\theta)$  $1$-extension on the vertices $v_1,v_2,v_3$ and the edge $e$.
In cases (1) and (2), if  $e\in E(G)$ then $G$ is obtained from $H$ by applying a modified $(\bZ_2,\theta)$  $1$-extension on the vertices $v_1,v_2,v_3$ and the edge $f$.
In case (3), if $e\in E(G)$ then $G$ is obtained from $H$ by applying a $(\bZ_2,\theta)$  $1$-extension on the vertices $v_1,v_2,v_3$ and the edge $f$. 
\endproof

\subsection{$(\bZ_2,\theta)$ fixed-vertex-to-$W_5$ extension.}
The  remaining graph extension in the construction scheme involves the wheel graph $W_5$ and action $\theta^*:\bZ_2\to \Aut(W_5)$  defined in Example \ref{WheelGraph}. The vertex in $W_5$ which is fixed by $s$ is denoted $v_0$. See Fig.~\ref{VertexG*} for an illustration.

\begin{definition}
\label{Csgtogtilde}
Let $G$ be a simple graph which is $\bZ_2$-symmetric with respect to $\theta$. Suppose that $W_5$ is a subgraph of $G$ and that the restriction of $\theta$ to $V(W_5)$ is the action $\theta^*$. Suppose further that there exists a graph $H$ with the following properties.
\begin{enumerate}[(i)]
\item
$V(G)=V(H)\cup V(W_5)$ and $V(H)\cap V(W_5)=\{v_0\}$ where $v_0$ is the $4$-valent vertex in $W_5$.
\item $H$ is $\bZ_2$-symmetric with respect to $\theta$. 
\item
$E(G)=E(W_5)\cup (E(H)\backslash\{wv_0\in E(H): w\in N(v_0)\})\cup \{w(\pi(w)):w\in N(v_0)\}$
where,
\begin{itemize}
	\item $N(v_0)$ is the set of vertices in $H$ which are adjacent to $v_0$, and,
	\item $\pi:N(v_0)\to V(W_5)$ is any map which satisfies $s(\pi(w))=\pi(sw)$ for all $w\in N(v_0)$.
\end{itemize}
\end{enumerate}
Then $G$ is said to be obtained from $H$ by applying a {\em $(\bZ_2,\theta)$ fixed-vertex-to-$W_5$ extension} (on the vertex $v_0$). 
\end{definition}

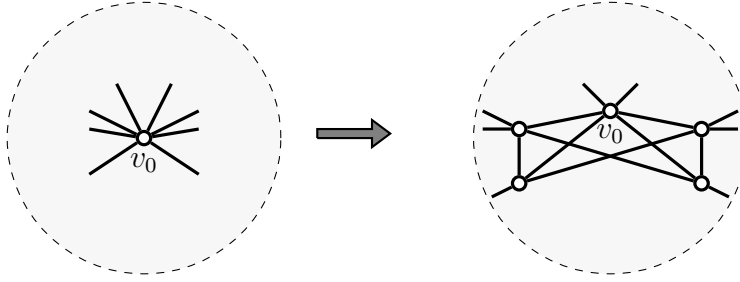
\begin{figure}[htp]
\begin{center}
\begin{tabular}{c}
\begin{tikzpicture}[very thick,scale=1.2]
\tikzstyle{every node}=[circle, draw=black, fill=white, inner sep=0pt, minimum width=5pt];
\filldraw[fill=black!03!white, draw=black, thin, dashed](0,0)circle(1.5cm);
\node[rectangle, draw=black!03!white,fill=black!03!white](l2) at (0,-0.27) {$v_0$};

\node (p1) at (0,0) {};

\draw (p1)  --  (-0.3,0.6);
\draw (p1)  --  (0.3,0.6);
\draw (p1)  --  (-0.6,0.3);
\draw (p1)  --  (0.6,0.3);
\draw (p1)  --  (-0.6,0.1);
\draw (p1)  --  (0.6,0.1);
\draw (p1)  --  (-0.6,-0.4);
\draw (p1)  --  (0.6,-0.4);

\filldraw[fill=black!50!white, draw=black, thick]
    (1.9,0) -- (2.5,0) -- (2.5,-0.1) -- (2.7,0.05) -- (2.5,0.2) -- (2.5,0.1) -- (1.9,0.1) -- (1.9,0);
\end{tikzpicture}
\hspace{0.8cm}
\begin{tikzpicture}[very thick,scale=1.2]
\tikzstyle{every node}=[circle, draw=black, fill=white, inner sep=0pt, minimum width=5pt];
\filldraw[fill=black!03!white, draw=black, thin, dashed](0,0)circle(1.5cm);
\node[rectangle, draw=black!03!white,fill=black!03!white](l2) at (0,0.05) {$v_0$};
\node[rectangle, draw=black!03!white,fill=black!03!white,above left](l1) at (-0.8,0.24) {};
\node[rectangle, draw=black!03!white,fill=black!03!white,above right](l2) at (0.8,0.24) {};
\node[rectangle, draw=black!03!white,fill=black!03!white,below](l1) at (-0.8,-0.65) {};
\node[rectangle, draw=black!03!white,fill=black!03!white,below](l2) at (0.9,-0.65) {};

\path (0,0.3) node (p1)  {} ;
\path (-1,0.1) node (p2)  {} ;
\path (1,0.1) node (p3)  {} ;
\path (-1,-0.5) node (p4)  {} ;
\path (1,-0.5) node (p6)  {} ;

\draw (p1)  --  (-0.3,0.6);
\draw (p1)  --  (0.3,0.6);
\draw (p2)  --  (-1.4,0.3);
\draw (p3)  --  (1.4,0.3);
\draw (p2)  --  (-1.4,0.1);
\draw (p3)  --  (1.4,0.1);
\draw (p4)  --  (-1.3,-0.65);
\draw (p6)  --  (1.3,-0.65);

\draw (p1)  --  (p2);
\draw (p1)  --  (p3);
\draw (p1)  --  (p4);
\draw (p1)  --  (p6);
\draw (p2)  --  (p4);
\draw (p3)  --  (p6);
\draw (p4)  --  (p3);
\draw (p6)  --  (p2);
\end{tikzpicture} 
\end{tabular}
\end{center}
\vspace{-0.2cm}
\caption{\emph{A $(\bZ_2,\theta)$ fixed-vertex-to-$W_5$ extension $G$ of a graph $H$.}}
\label{VertexG*}
\end{figure}

\begin{lemma}
\label{3VertexLemma}
Let $(G,\theta)$ be an admissible pair with minimum vertex degree 3 and let $v$ be a $3$-valent vertex in $G$.
Suppose $N(v)$ and $N(sv)$ intersect in three vertices and that $G\not=W_5$.
Then  there exists a graph $H$ with the following properties. 
\begin{enumerate}[(i)]
\item $G$ is obtained from $H$ by applying  either a $(\bZ_2,\theta)$ $1$-extension, a modified $(\bZ_2,\theta)$ $1$-extension or a $(\bZ_2,\theta)$ fixed-vertex-to-$W_5$ extension.
\item $(H,\theta)$ is an admissible pair. 
\end{enumerate}
\end{lemma}

\proof
By Lemma \ref{NbdLemma}, $N(v)$ must contain the fixed vertex $v_0$ and so
$N(v)=\{v_0,v_1,sv_1\}$, say. 
The graph $G$ may be expressed as a union of two edge-disjoint $\bZ_2$-symmetric spanning trees $G_1$ and $G_2$. Note that the edges $vv_1$ and $v(sv_1)$ cannot both belong to the same spanning tree since, otherwise, either $G_1$ or $G_2$ would contain the $4$-cycle $vv_1,v_1(sv),s(vv_1),(sv_1)v$.
Since $v_1$ and $sv_1$ can be relabeled as $sv_1$ and $v_1$ respectively, it may be assumed that $vv_0$ and $vv_1$ are both contained in $G_1$.
If $v_0v_1\notin E(G)$ then $\S(v)$ is represented by the first graph in 
Fig. \ref{Degree3VertexD}.
Let $H$ be the graph obtained  by adjoining the edges $e=v_0v_1$ and $se$ to $G\backslash\{v,sv\}$ and let $H_1=(G_1\cap H)\cup\{e,se\}$ and $H_2=G_2\cap H$. 
Then $G$ is obtained from $H$ by applying a $(\bZ_2,\theta)$ $1$-extension on the vertices $v_0$, $v_1$, $sv_1$ and the edge $se$.
If $v_0v_1\in E(G)$ then $\S(v)$ is a copy of the wheel graph $W_5$ and is represented by the second graph in Fig. \ref{Degree3VertexD}.
Note that $\S(v)$ is $\bZ_2$-symmetric with respect to $\theta$ and no edge of $\S(v)$ is fixed by $s$. Thus the restriction of $\theta$ to $\S(v)$ is the unique action $\theta^*$ defined in Example \ref{WheelGraph}.
There are two cases to consider.
\begin{figure}[htp]
\begin{center}
\begin{tabular}{c}

\begin{tikzpicture}[very thick,scale=0.8]
\tikzstyle{every node}=[circle, draw=black, fill=white, inner sep=0pt, minimum width=5pt];
\filldraw[fill=black!03!white, draw=black, thin, dashed](0,0)circle(1.62cm);
\node[rectangle, draw=black!03!white,fill=black!03!white](l1) at (-0.65,1.03) {$v_1$};
\node[rectangle, draw=black!03!white,fill=black!03!white](l2) at (0.65,1.05) {$sv_1$};
\node[rectangle, draw=black!03!white,fill=black!03!white](l1) at (0,-1) {$v_0$};
\node (p1) at (-0.65,0.75) {};
\node (p2) at (0.65,0.75) {};
\node (p5) at (0,-0.6) {};
\node (p7) at (-2,0) {};
\node (p8) at (2,0) {};
\node[rectangle, draw=white,fill=white,below right](l1) at (-2.25,-0.2) {$v$};
\node[rectangle, draw=white,fill=white,below left](l2) at (2.3,-0.2) {$sv$};
\draw(p7)--(p1);
\draw[lightgray](p7)--(p2);
\draw(p7)--(p5);
\draw(p8)--(p2);
\draw[lightgray](p8)--(p1);
\draw(p8)--(p5);
\end{tikzpicture}
\hspace{0.7cm}
\begin{tikzpicture}[very thick,scale=0.8]
\tikzstyle{every node}=[circle, draw=black, fill=white, inner sep=0pt, minimum width=5pt];
\filldraw[fill=black!03!white, draw=black, thin, dashed](0,0)circle(1.62cm);
\node[rectangle, draw=black!03!white,fill=black!03!white](l1) at (-0.65,1.03) {$v_1$};
\node[rectangle, draw=black!03!white,fill=black!03!white](l2) at (0.65,1.05) {$sv_1$};
\node[rectangle, draw=black!03!white,fill=black!03!white](l1) at (0,-1) {$v_0$};
\node (p1) at (-0.65,0.75) {};
\node (p2) at (0.65,0.75) {};
\node (p5) at (0,-0.6) {};
\node (p7) at (-2,0) {};
\node (p8) at (2,0) {};
\node[rectangle, draw=white,fill=white,below right](l1) at (-2.25,-0.2) {$v$};
\node[rectangle, draw=white,fill=white,below left](l2) at (2.3,-0.2) {$sv$};
\draw(p7)--(p1);
\draw[lightgray](p7)--(p2);
\draw(p7)--(p5);
\draw(p8)--(p2);
\draw[lightgray](p8)--(p1);
\draw(p8)--(p5);
\draw[lightgray](p5)--(p1);
\draw[lightgray](p5)--(p2);
\end{tikzpicture}
\end{tabular}
\end{center}
\vspace{-0.3cm}
\caption{$N(v)$ and $N(sv)$  intersect in three vertices.}
\label{Degree3VertexD}
\end{figure}
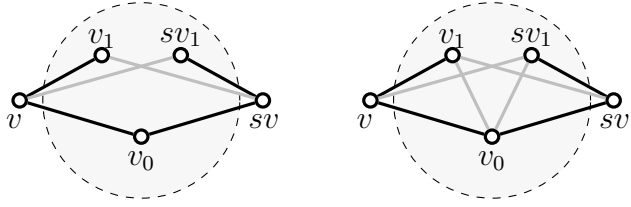

{\bf Case (1)}
Suppose there exists a vertex in $G\backslash \S(v)$ which is adjacent to two vertices in $\S(v)$. Then $\S(v)$ is not contractible.
If $G$ contains no other symmetric copies of $W_5$ then by counting degrees of vertices it follows that $|E|\geq 2|V|-1$. This is a contradiction since $G$ is an edge disjoint union of two spanning trees. More generally, if $G$ contains exactly $k$ non-contractible symmetric copies of $W_5$ and every vertex of degree $3$ is contained in one of these copies then $|E|\geq 2|V|-2+k$. This is a contradiction and so there must exist a vertex $w$ in $G$ with  degree $3$ such that the symmetric neighbourhood of $w$ is not a copy of $W_5$. In this case, $G$ may obtained from an admissible pair $(H,\theta)$ by either a $1$-extension or a modified $1$-extension.

{\bf Case (2)}
Suppose  that no vertex of $G\backslash \S(v)$ is  adjacent to two vertices in
$\S(v)$ and let $H$ be the graph obtained by contracting $\S(v)$ to $v_0$.
Then $G$ may be obtained from $H$ by applying a $(\bZ_2,\theta)$ fixed-vertex-to-$W_5$ extension on the vertex $v_0$.
Let $H_1$ and $H_2$ be the graphs obtained from $G_1$ and $G_2$ respectively as a result of this contraction operation. 
The graphs $H_1$ and $H_2$ are edge-disjoint $\bZ_2$-symmetric spanning trees in $H$ and it follows that $(H,\theta)$ is an admissible pair. 
\endproof

\begin{remark}
The role of the graph $W_5$ in the construction scheme is analogous to the role played by the complete graph $K_4$  in  constructing $(2,2)$-tight graphs (see \cite{NOP}). 
The $(\bZ_2,\theta)$ fixed-vertex-to-$W_5$ extension (described above) is comparable to the {\em vertex-to-$K_4$ move}.   
\end{remark}

\subsection{Construction scheme for admissible pairs.}
\label{Construction}
The four $(\bZ_2,\theta)$ graph extensions described  above  are referred to as {\em allowable}
and the notation $H\to G$ indicates that a graph $G$ is obtained from $H$ by an allowable graph extension.
To summarise, the allowable graph extensions are:
$(\bZ_2,\theta)$ $0$-extensions,
$(\bZ_2,\theta)$ $1$-extensions,
modified $(\bZ_2,\theta)$ $1$-extensions and
$(\bZ_2,\theta)$ fixed-vertex-to-$W_5$ extensions.

\begin{theorem}
\label{ConstructionScheme}
Let $G$ be a finite simple graph and let $\theta:\bZ_2\to \Aut(G)$ be an action of the group $\bZ_2=\langle s\rangle$.
If $G$ is expressible as a union of two edge-disjoint spanning trees,  both of which are $\bZ_2$-symmetric with respect to $\theta$,
and no edge of $G$ is fixed by $s$
then there exists a construction chain,
\[W_5=G^1\to G^2 \to\cdots \to G^n = G,\]
such that for each $k=1,2,\ldots,n-1$,
\begin{enumerate}[(i)]
\item $G^k\to G^{k+1}$ is an allowable $(\bZ_2,\theta)$ graph extension, and,
\item $G^k$ is expressible as a union of two edge-disjoint spanning trees,  both of which are $\bZ_2$-symmetric with respect to $\theta$, and no edge of $G^k$ is fixed by $s$.
\end{enumerate}
\end{theorem}

\proof
Note that $W_5$ is the only graph on five or fewer vertices which admits a $\bZ_2$-action that satisfies the hypothesis of the theorem. 
Suppose, to obtain a proof by induction, that $|V(G)|>5$ and that the statement of the theorem holds for all admissible pairs $(\tilde{G},\tilde{\theta})$ for which $|V(\tilde{G})|<|V(G)|$.
Since $|E(G)|=2(|V(G)|-1)$ it follows that $G$ must contain a vertex $v$ of degree $2$ or $3$.
If $\deg(v)=2$ then, by Lemma \ref{Degree2VertexLemma},  there exists an admissible pair $(H,\theta)$ together with a $(\bZ_2,\theta)$ $0$-extension $H\to G$. 
If $\deg(v)=3$ then $\S(v)$  satisfies the conditions of one of the Lemmas \ref{2VertexLemma}, \ref{0VertexLemma}, \ref{1VertexLemma}, and \ref{3VertexLemma}.
Thus there exists an admissible pair $(H,\theta)$ together with an allowable graph extension $H\to G$. 
In each case, $H$ has fewer vertices than $G$ and so there exists a construction chain for $H$.
This establishes the existence of a construction chain for $G$ and so the induction step is complete.
\endproof


\section{$\bZ_2$-symmetric frameworks in $(\bR^2,\|\cdot\|_\P)$}
\label{sec:suffcon}
A bar-joint framework $(G,p)$ in $\bR^2$ consists of a finite simple graph $G$ and a point $p=(p(v))_{v\in V}$ such that the components $p(v)$ are distinct points in $\bR^2$.
Let $\|\cdot\|_\P$ be a norm on $\bR^2$ with the property that the unit ball $\P$ is a quadrilateral (eg. the $\ell^1$ or $\ell^\infty$ norm). 
A bar-joint framework is $\Gamma$-symmetric in $(\bR^2,\|\cdot\|_\P)$
if there exists a group action $\theta:\Gamma\to\Aut(G)$ and a group representation $\tau:\Gamma\to\GL(\bR^2)$ such that for each $\gamma\in \Gamma$, $\tau(\gamma)$ is an isometry of $(\bR^2,\|\cdot\|_\P)$ and $\tau(\gamma)p(v)=p(\gamma v)$ for all vertices $v\in V(G)$.
The elements of $\Gamma$ are called symmetry operations of $(G,p)$
and $\Gamma$ is called a symmetry group of $(G,p)$.
For each facet $F$ of $\P$, let $[F]=\{F,-F\}$.

\begin{definition}
Let $\|\cdot\|_\P$ be a norm on $\bR^2$ where the unit ball $\P$ is a quadrilateral and let $\tau:\Gamma\to \GL(\bR^2)$ be a group representation.
A group element $\gamma\in \Gamma$ {\em preserves the facets} of $\P$ if $\tau(\gamma)F\in [F]$  for each facet $F$ of $\P$. 
\end{definition}

A bar-joint framework is {\em well-positioned} in $(\bR^2,\|\cdot\|_\P)$ if for each edge $vw$, the normalised vector $\frac{p(v)-p(w)}{\|p(v)-p(w)\|_\P}$ 
is contained in exactly one facet $F$ of $\P$. The pair $[F]$ corresponding to this unique facet is referred to as the framework colour of the edge $vw$.
Denote by $G_F$ the {\em monochrome} subgraph of $G$ spanned by edges with framework colour $[F]$, provided such edges exist. 

\begin{lemma}
\label{Compatible}
Let $(G,p)$ be a well-positioned and $\Gamma$-symmetric bar-joint framework in $(\bR^2,\|\cdot\|_\P)$. 
If each symmetry operation $\gamma\in \Gamma$ preserves the facets of $\P$
then the monochrome subgraphs of $G$ are $\Gamma$-symmetric.
\end{lemma}

\proof
Let $\gamma\in \Gamma$ and suppose that $vw$ is an edge of $G$ with framework colour $[F]$.
Then $p(v)-p(w)$ is contained in the conical hull of either $F$ or $-F$.
Since $\gamma$ preserves the facets of $p$, $p(\gamma v)-p(\gamma w)=\tau(\gamma)(p(v)-p(w))$ is also contained in the conical hull of either $F$ or $-F$.
Thus $\gamma(vw)$ has the same framework colour as $vw$ and so
$G_F$ is $\Gamma$-symmetric.
\endproof

In \cite{kit-sch}, it is shown (for general norms) that $(G,p)$ is well-positioned if and only if the rigidity map $f_G:(\bR^2)^{|V(G)|}\to \bR^{|E(G)|}$, $(x(v))_{v\in V(G)} \mapsto(\|x(v)-x(w)\|_\P)_{vw\in E(G)}$ is differentiable at $p$. The elements of $\ker df_G(p)$ are called {\em infinitesimal flexes} of $(G,p)$. 
A collection of continuous paths $\alpha_{x}:(-\delta,\delta)\to \bR^2$, $x\in \bR^2$, is called a continuous rigid motion of $(\bR^2,\|\cdot\|_\P)$ if $\alpha_x(0)=x$ for all  $x\in \bR^2$
and $\|\alpha_x(t)-\alpha_y(t)\|_\P=\|x-y\|_\P$ for all $t\in(-\delta,\delta)$ and all $x,y\in \bR^2$.
An infinitesimal flex $u=(u(v))_{v\in V(G)}$ is regarded as {\em trivial} if there exists a continuous rigid motion such that $u(v)=\alpha_{p(v)}'(0)$ for all $v\in V(G)$.
If every infinitesimal flex of $(G,p)$ is trivial then $(G,p)$ is {\em infinitesimally rigid}.
A framework is {\em isostatic} if it is infinitesimally rigid and no proper spanning subframework is infinitesimally rigid. 
The rigidity of bar-joint frameworks with respect to norms for which the unit ball is a convex polytope is developed in \cite{kitson}. In particular, it is shown that infinitesimal rigidity is equivalent to continuous rigidity for well-positioned frameworks. 
The following theorem is proved in \cite{kit-pow}.

\begin{theorem}
\label{SpanningTreeThm}
Let $\|\cdot\|_\P$ be a norm on $\bR^2$ where the unit ball $\P$ is a quadrilateral.
Let $(G,p)$ be a well-positioned bar-joint framework in $(\bR^2,\|\cdot\|_\P)$.
Then $(G,p)$ is isostatic if and only if $G$ is an edge-disjoint union of two  monochrome spanning trees. 
\end{theorem}

The following lemma involves the wheel graph $W_5$ and action $\theta^*$ of Example \ref{WheelGraph}.

\begin{lemma}
\label{W5Placement}
Let $\|\cdot\|_\P$ be a norm on $\bR^2$ for which the unit ball $\P$ is a quadrilateral and let $\tau:\bZ_2\to\GL(\bR^2)$ be a representation of $\bZ_2=\langle s\rangle$ such that $s$ preserves the facets of $\P$.
Then there exists a point $p$ such that $(W_5,p)$ is isostatic and $\bZ_2$-symmetric with respect to $\theta^*$ and $\tau$.
\end{lemma}

\proof
Let $V(W_5)=\{v_0,v_1,sv_1,v_2,sv_2\}$ be the vertices of $W_5$ and let $\pm F_1$ and $\pm F_2$ be the facets of $\P$.
Choose points $x_1$ and $x_2$ in the relative interiors 
$F_1$ and $F_2$ respectively and let $y$ be an extreme point of $\P$.
Let $p(v_0)$ be any point in $\bR^2$ which is fixed by $\tau(s)$
and choose $p(v_1)$ to be a point on the line $L_1=\{p(v_0)+tx_1:t\in\bR\}$ which is distinct from $p(v_0)$.
Define $p(sv_1)=\tau(s)p(v_1)$.
Then the edge $v_0v_1$ has framework colour $[F_1]$ and, since $s$ preserves the facets of $\P$, the edge $s(v_0v_1)$ also has framework colour $[F_1]$.
Let $a$ be the point of intersection of the lines $L_2=\{p(v_0)+tx_2:t\in\bR\}$  and $L_3=\{p(sv_1)+ty:t\in\bR\}$ and let $B(a,r)$ be an open ball centred at $a$ with radius $r>0$.
Choose $p(v_2)$ to be a point in $B(a,r)$ which is distinct from $p(v_0),p(v_1),p(sv_1)$ and such that $p(v_2)\not=\tau(s)p(v_2)$. 
Note that if $r$ is sufficiently small then $v_0v_2$ has framework 
colour $[F_2]$.
Suppose $v_1v_2$ has framework colour $[F_1]$.
Then $p(v_2)$ may be chosen so that the edge $v_2(sv_1)$ has framework colour $[F_2]$. Similarly, if $v_1v_2$ has framework colour $[F_2]$
then $p(v_2)$ may be chosen so that $v_2(sv_1)$ has framework colour $[F_1]$.
Define $p(sv_2)=\tau(s)p(v_2)$.
Then $(W_5,p)$ is  well-positioned in $(\bR^2,\|\cdot\|_\P)$  and  $\bZ_2$-symmetric with respect to $\theta^*$ and $\tau$.
Moreover, the monochrome subgraphs $(W_5)_{F_1}$ and $(W_5)_{F_2}$ are spanning trees in $W_5$ and so, by Theorem \ref{SpanningTreeThm}, $(W_5,p)$ is  isostatic.
\endproof

A symmetry operation $\gamma\in \Gamma$ is called a {\em reflection} if there exists a rank one projection $P$ on $\bR^2$ such that $\tau(\gamma)=I-2P$, and,
a {\em half-turn rotation} if $\tau(\gamma)$ is the rotation of $\bR^2$ by $\pi$ about the origin. 
 A symmetry group which is generated by a reflection (respectively, a half-turn rotation) is denoted by $\C_s$ (respectively, by $\C_2$).

\begin{example}
Two isostatic placements of $W_5$ in $(\bR^2,\|\cdot\|_\infty)$ are indicated in Fig. \ref{fig:wheel} with induced monochrome symmetric spanning trees  indicated in black and gray respectively. The framework on the left is $\C_2$-symmetric and the framework on the right is $\C_s$-symmetric with respect to reflection in the $y$-axis. 

\vspace{-0.3cm}

\begin{figure}[htp]
\begin{center}
\begin{tikzpicture}[very thick,scale=0.6]
\tikzstyle{every node}=[circle, draw=black, fill=white, inner sep=0pt, minimum width=5pt];

\node[rectangle, draw=white,fill=white,below](l2) at (0,-0.25) {$v_0$};
\node[rectangle, draw=white,fill=white,above left](l1) at (-1.15,0.85) {$v_1$};
\node[rectangle, draw=white,fill=white,above right](l2) at (1.18,1) {$sv_2$};
\node[rectangle, draw=white,fill=white,below left](l1) at (-1.15,-0.8) {$v_2$};
\node[rectangle, draw=white,fill=white,below right](l2) at (1.15,-0.6) {$sv_1$};

\path (0,0.1) node (p1)  {} ;
\path (-1,-1) node (p2)  {} ;
\path (1,-0.8) node (p3)  {} ;
\path (1,1.2) node (p4)  {} ;
\path (-1,1) node (p5)  {} ;

\draw[lightgray] (p1)  --  (p2);
\draw (p1)  --  (p3);
\draw[lightgray] (p1)  --  (p4);
\draw (p1)  --  (p5);
\draw (p2)  --  (p3);
\draw[lightgray] (p3)  --  (p4);
\draw (p4)  --  (p5);
\draw[lightgray] (p5)  --  (p2);

\end{tikzpicture} 
\hspace{18mm}
\begin{tikzpicture}[very thick,scale=0.8]
\tikzstyle{every node}=[circle, draw=black, fill=white, inner sep=0pt, minimum width=5pt];

\node[rectangle, draw=white,fill=white](l2) at (0,0.98) {$v_0$};
\node[rectangle, draw=white,fill=white,above left](l1) at (-1.3,0) {$v_1$};
\node[rectangle, draw=white,fill=white,above right](l2) at (1.3,0) {$sv_1$};
\node[rectangle, draw=white,fill=white,below](l1) at (-1.68,-0.8) {$v_2$};
\node[rectangle, draw=white,fill=white,below](l2) at (1.72,-0.8) {$sv_2$};

\path (0,0.7) node (p1)  {} ;
\path (-1.2,0) node (p2)  {} ;
\path (1.2,0) node (p3)  {} ;
\path (-1.3,-0.9) node (p4)  {} ;
\path (1.3,-0.9) node (p6)  {} ;

\draw (p1)  --  (p2);
\draw (p1)  --  (p3);
\draw[lightgray] (p1)  --  (p4);
\draw[lightgray] (p1)  --  (p6);
\draw[lightgray] (p2)  --  (p4);
\draw[lightgray] (p3)  --  (p6);
\draw (p4)  --  (p3);
\draw (p6)  --  (p2);

\end{tikzpicture} 
     \end{center}
\vspace{-0.3cm}
     \caption{Isostatic placements of the wheel graph $W_5$ with induced $\bZ_2$-symmetric monochrome spanning trees.}
\label{fig:wheel}
\end{figure}
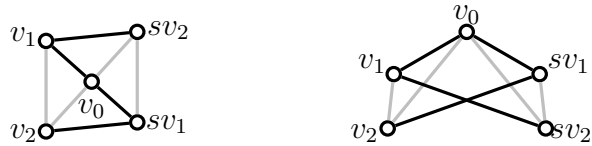
\end{example}

\subsection{Rigidity preservation for allowable graph extensions}
Let $\|\cdot\|_\P$ be a norm on $\bR^2$ for which the unit ball $\P$ is a quadrilateral with facets $\pm F_1$ and $\pm F_2$.
Let $G$ be a finite simple graph and let $\theta:\bZ_2\to \Aut(G)$ be an action of the group $\bZ_2=\langle s\rangle$. 

\begin{proposition}
\label{RigidityPreservation1}
Suppose there exists a graph $H$ with the following properties.
\begin{enumerate}[(i)]
\item $G$ is obtained from $H$ by either a $(\bZ_2,\theta)$ $0$-extension or a  $(\bZ_2,\theta)$ $1$-extension.
\item There exists a faithful representation $\tau:\bZ_2\to \GL(\bR^2)$ and a point $p_H$ such that  $(H,p_H)$ is a well-positioned and isostatic  bar-joint  framework in $(\mathbb{R}^2,\|\cdot\|_\P)$ which is $\bZ_2$-symmetric with respect to $\theta$ and $\tau$, where $s$ preserves the facets of $\P$.
\end{enumerate}
Then there exists $p$ such that $(G,p)$ is a well-positioned and isostatic bar-joint framework in $(\mathbb{R}^2,\|\cdot\|_\P)$ which is $\bZ_2$-symmetric with respect to $\theta$ and $\tau$. 
\end{proposition}

\proof
Note that $V(G)=V(H)\cup \{v,sv\}$ for some $v,sv\notin V(H)$.
Define $p(w)=p_H(w)$ for all $w\in V(H)$ and, once $p(v)$ has been chosen, define 
$p(sv)=\tau(s)(p(v))$. 
Then $(G,p)$ is $\bZ_2$-symmetric with respect to $\theta$ and $\tau$. 
It remains to specify $p(v)$.
Choose points $x_1$ and $x_2$ in the relative interiors of $F_1$ and $F_2$ respectively. 

If $G$ is  obtained by a $(\bZ_2,\theta)$ $0$-extension on vertices $v_1, v_2\in V(H)$ then let $a\in \bR^2$ be the point of intersection of the lines $L_1=\{p(v_1)+tx_1:t\in \bR\}$ and $L_2=\{p(v_2)+tx_2:t\in \bR\}$ and
let $B(a,r)$ be an open  ball with centre $a$ and radius $r>0$.
Choose $p(v)$ to be any point in $B(a,r)\backslash p(V(H))$ which is not 
fixed by $\tau(s)$. If $r$ is sufficiently small then 
the induced framework colourings for $vv_1$ and $vv_2$ are $[F_1]$ and $[F_2]$ respectively. 
Thus, $(G,p)$ is well-positioned  with monochrome subgraphs  $G_{F_1}=H_{F_1}\cup \{vv_1,s(vv_1)\}$ and $G_{F_2}=H_{F_2}\cup \{vv_2,s(vv_2)\}$. 

\begin{figure}[htp]
\begin{center}
\begin{tikzpicture}[very thick,scale=1.4]
\tikzstyle{every node}=[circle, draw=black, fill=white, inner sep=0pt, minimum width=5pt];
\filldraw[fill=black!03!white, draw=black, thin, dashed](-1.6,0)circle(0.35cm);
\node[rectangle, draw=white,fill=white](l1) at (-0.45,-0.45) {\small$p(v_2)$};
\node[rectangle, draw=white,fill=white](l2) at (0.65,-0.45) {\small$p(sv_2)$};
\node[rectangle, draw=white,fill=white](l3) at (-0.5,0.6) {\small$p(v_1)$};
\node[rectangle, draw=white,fill=white](l4) at (0.45,0.6) {\small$p(sv_1)$};
\node (p1) at (-0.5,0.4) {};
\node (p2) at (0.5,0.4) {};
\node (p3) at (-0.72,-0.67) {};
\node (p4) at (0.72,-0.67) {};

\node[rectangle, draw=black!03!white,fill=black!03!white](l4) at (-1.75,0) {$a$};
\node[rectangle, draw=white,fill=white](l4) at (-2.3,-0.34) {\small$B(a,r)$};
\filldraw[fill=black, draw=black, thin](-1.6,0)circle(0.04cm);
\draw[dashed] (p1)--(-1.6,0);
\draw[dashed] (p3)--(-1.6,0);
\node[rectangle, draw=white,fill=white](l4) at (-0.93,0.04) {\small $L_1$};
\node[rectangle, draw=white,fill=white](l4) at (-1.19,-0.55) {\small $L_2$};
\end{tikzpicture}
\hspace{1.2cm}
\begin{tikzpicture}[very thick,scale=1.4]
\tikzstyle{every node}=[circle, draw=black, fill=white, inner sep=0pt, minimum width=5pt];
\filldraw[fill=black!03!white, draw=black, thin, dashed](-1.6,0)circle(0.35cm);

\node[rectangle, draw=white,fill=white](l1) at (-0.45,-0.45) {\small$p(v_2)$};
\node[rectangle, draw=white,fill=white](l2) at (0.42,-0.45) {\small$p(sv_2)$};
\node[rectangle, draw=white,fill=white](l3) at (-0.5,0.6) {\small$p(v_1)$};
\node[rectangle, draw=white,fill=white](l4) at (0.4,0.6) {\small$p(sv_1)$};
\node (p1) at (-0.5,0.4) {};
\node (p2) at (0.5,0.4) {};
\node (p3) at (-0.72,-0.67) {};
\node (p4) at (0.72,-0.67) {};
\node (p6) at (1.44,0.2) {};
\node (p7) at (-1.44,0.2) {};

\node[rectangle, draw=black!03!white,fill=black!03!white](l4) at (-1.63,-0.05) {\small$p(v)$};
\node[rectangle, draw=white,fill=white](l4) at (1.7,-0.03) {\small$p(sv)$};
\draw(p7)--(p1);
\draw[lightgray](p7)--(p3);
\draw(p6)--(p2);
\draw[lightgray](p6)--(p4);
\end{tikzpicture}
\end{center}
\caption{\emph{Constructing a placement for a $(\bZ_2,\theta)$ $0$-extension.}}
\label{1ExtPlacement}
\end{figure}
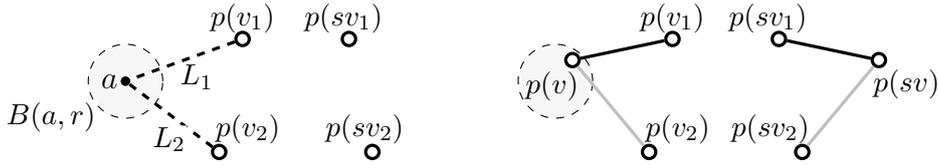

If $G$ is  obtained by a $(\bZ_2,\theta)$ $1$-extension on vertices $v_1,v_2,v_3\in V(H)$ and the edge $e=v_1v_2\in E(H)$ 
then, without loss of generality, suppose $e\in E(H_{F_1})$. 
Let $a\in \bR^2$ be the point of intersection of the line $L_1$ which passes through the points $p(v_1)$ and $p(v_2)$ with $L_2=\{p(v_3)+tx_2:t\in\bR\}$.
Let $B(a,r)$ be the open  ball with centre $a$ and radius $r>0$ and choose $p(v)$ to be a point in $B(a,r)\backslash p(V(H))$ which is not fixed by $\tau(s)$.
If $r$ is sufficiently small then $vv_1$ and $vv_2$ have induced framework colour $[F_1]$ and $vv_3$ has framework colour $[F_2]$.
Then $(G,p)$ is well-positioned with monochrome subgraphs  $G_{F_1}=(H_{F_1}\backslash\{e,se\})\cup \{vv_1,vv_2,s(vv_1),s(vv_2)\}$ and $G_{F_2}=H_{F_2}\cup \{vv_3,s(vv_3)\}$.

In both cases, the monochrome subgraphs $G_{F_1}$ and $G_{F_2}$ induced by $p$
are spanning trees in $G$. Thus, by Theorem \ref{SpanningTreeThm}, $(G,p)$ is isostatic. 
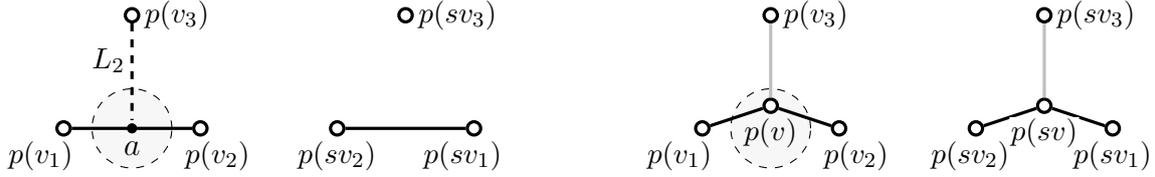
\begin{figure}[htp]
\begin{center}
\begin{tikzpicture}[very thick,scale=1.5]
\tikzstyle{every node}=[circle, draw=black, fill=white, inner sep=0pt, minimum width=5pt];

\node (p1) at (-2.6,0) {};
\node (p2) at (-1.4,0) {};
\node (p3) at (-2,1) {};
\node (p4) at (-2,0) {};
\node (p5) at (1,0) {};
\node (p6) at (-0.2,0) {};
\node (p7) at (0.4,1) {};

\filldraw[fill=black!03!white, draw=black, thin, dashed](p4)circle(0.35cm);
\draw[dashed] (p3)--(p4);
\draw(p1)--(p2);
\draw(p5)--(p6);

\node[rectangle, draw=white,fill=white,below](l1) at (-1.25,-0.1) {\small $p(v_2)$};
\node[rectangle, draw=white,fill=white,below](l2) at (-0.2,-0.1) {\small $p(sv_2)$};
\node[rectangle, draw=white,fill=white,below](l3) at (-2.8,-0.1) {\small$p(v_1)$};
\node[rectangle, draw=white,fill=white,below](l4) at (0.9,-0.1) {\small$p(sv_1)$};
\node[rectangle, draw=white,fill=white,right](l3) at (-1.9,1) {\small$p(v_3)$};
\node[rectangle, draw=white,fill=white,right](l4) at (0.5,1) {\small$p(sv_3)$};
\node[rectangle, draw=black!03!white,fill=black!03!white,below](l4) at (-2,-0.1) {$a$};
\filldraw[fill=black, draw=black, thin](p4)circle(0.04cm);

\node[rectangle, draw=white,fill=white, left](l4) at (-2.05,0.6) {\small $L_2$};
\end{tikzpicture}
\hspace{1.6cm}
\begin{tikzpicture}[very thick,scale=1.5]
\tikzstyle{every node}=[circle, draw=black, fill=white, inner sep=0pt, minimum width=5pt];

\node (p1) at (-2.6,0) {};
\node (p2) at (-1.4,0) {};
\node (p3) at (-2,1) {};
\node (p4) at (-2,0) {};
\node (p5) at (1,0) {};
\node (p6) at (-0.2,0) {};
\node (p7) at (0.4,1) {};
\filldraw[fill=black!03!white, draw=black, thin, dashed](p4)circle(0.35cm);
\node (p8) at (-2,0.2) {};
\node (p9) at (0.4,0.2) {};

\draw(p1)--(p8);
\draw(p2)--(p8);
\draw[lightgray](p3)--(p8);
\draw(p5)--(p9);
\draw(p6)--(p9);
\draw[lightgray](p7)--(p9);

\node[rectangle, draw=white,fill=white,below](l1) at (-1.25,-0.1) {\small$p(v_2)$};
\node[rectangle, draw=white,fill=white,below](l2) at (-0.25,-0.1) {\small$p(sv_2)$};
\node[rectangle, draw=white,fill=white,below](l3) at (-2.8,-0.1) {\small$p(v_1)$};
\node[rectangle, draw=white,fill=white,below](l4) at (1,-0.1) {\small$p(sv_1)$};
\node[rectangle, draw=white,fill=white,right](l3) at (-1.9,1) {\small$p(v_3)$};
\node[rectangle, draw=white,fill=white,right](l4) at (0.5,1) {\small$p(sv_3)$};
\node[rectangle, draw=black!03!white,fill=black!03!white,below](l4) at (-2,0.1) {\small $p(v)$};
\node[rectangle, draw=white,fill=white,below](l4) at (0.4,0.1) {\small$p(sv)$};

\end{tikzpicture}
\end{center}
\caption{\emph{Constructing a placement for a $(\bZ_2,\theta)$ $1$-extension.}}
\label{2ExtPlacement}
\end{figure}
\endproof

\begin{proposition}
\label{RigidityPreservation2}
Suppose there exists a graph $H$ with the following properties.
\begin{enumerate}[(i)]
\item $G$ is obtained from $H$ by a modified $(\bZ_2,\theta)$ $1$-extension.
\item There exists a faithful representation $\tau:\bZ_2\to \GL(\bR^2)$ and a point $p_H$ such that  $(H,p_H)$ is a well-positioned and isostatic  bar-joint  framework in $(\mathbb{R}^2,\|\cdot\|_\P)$ which is $\bZ_2$-symmetric with respect to $\theta$ and $\tau$, where $s$ preserves the facets of $\P$.
\end{enumerate}
Then there exists $p$ such that $(G,p)$ is a well-positioned and isostatic bar-joint framework in $(\mathbb{R}^2,\|\cdot\|_\P)$ which is $\bZ_2$-symmetric with respect to $\theta$ and $\tau$. 
\end{proposition}

\proof
Suppose $G$ is  obtained by  a modified $(\bZ_2,\theta)$ $1$-extension on vertices $v_1,v_2,v_3\in V(H)$ and the edge $e=v_1(sv_2)\in E(H)$.  
Then  $V(G)=V(H)\cup \{v,sv\}$ for some $v,sv\notin V(H)$.
Define $p(w)=p_H(w)$ for all $w\in V(H)$.
It may be assumed that the induced framework colour of $e$ is $[F_1]$.
Let $y_1$ be an extreme point of $\P$ and choose a point $x_1$ in the relative interior of $F_1$.
Since $H_{F_1}$ is a spanning tree in $H$, there must exist a path $P$ in $H_{F_1}\backslash\{e\}$ joining $v_3$ to either $v_1$, or,  $sv_2$. 

{\bf Case $(1)$}. Suppose  $P$ joins $v_3$ to $v_1$.
Let $a\in \bR^2$ be the point of intersection of the lines $L_1=\{p(v_2)+tx_1:t\in \bR\}$ and $L_2=\{p(v_3)+ty_1:t\in \bR\}$ and let $B(a,r)$ be an open  ball with centre $a$ and radius $r>0$.
Choose $p(v)$ to be a point in $B(a,r)\backslash p(V(H))$ which is not fixed by $\tau(s)$.
If $r$ is sufficiently small then the induced framework colour of $vv_2$ is $[F_1]$.
Set $p(sv)=\tau(s)(p(v))$.
Then $(G,p)$ is both well-positioned  and $\bZ_2$-symmetric with respect to $\theta$ and $\tau$.
Note that since $a\in L_2$, it is possible to choose $p(v)$ in the open ball $B(a,r)$ such that
$vv_3$ has framework colour either $[F_1]$ or $[F_2]$.
If the edge $vv_1$ has framework colour $[F_1]$ then choose $p(v)$ such that
$vv_3$ has framework colour $[F_2]$.
If the edge $vv_1$ has framework colour $[F_2]$ then choose $p(v)$ such that
$vv_3$ has framework colour $[F_1]$.
Then $(G,p)$ is well-positioned and $\bZ_2$-symmetric.
Also, $G_{F_1}$ and $G_{F_2}$ are spanning subgraphs of $G$ with $|V(G)|-1$ edges.
Clearly, $G_{F_2}$ is a tree.
To see that $G_{F_1}$ is a tree it remains to show that it is connected.
There exists a path $P_v$ in $G_{F_1}$ containing $v,v_1,v_2$ since either $vv_1$ and $vv_2$, or, $vv_2$ and $vv_3$ are edges of $G_{F_1}$ and $P$ joins $v_3$ to $v_1$.
Similarly, $s(P_v)$ is a path in $G_{F_1}$ containing $sv,sv_1,sv_2$. 
All remaining vertices in $G\backslash\{v,sv\}$ can be joined to either $v_1$ or $sv_2$ by a path in $G_{F_1}$.
In particular, there exists a path $P_0$ in $G_{F_1}$ from $v_0$ to either $v_1$ or $sv_2$. 
Thus $P_v\cup s(P_v)\cup P_0\cup s(P_0)$ contains a path in $G_{F_1}$ joining $v_1$ to $sv_2$. It follows that every vertex of $G$ can be joined to $v_1$ by a path in $G_{F_1}$ and so $G_{F_1}$ is connected. 

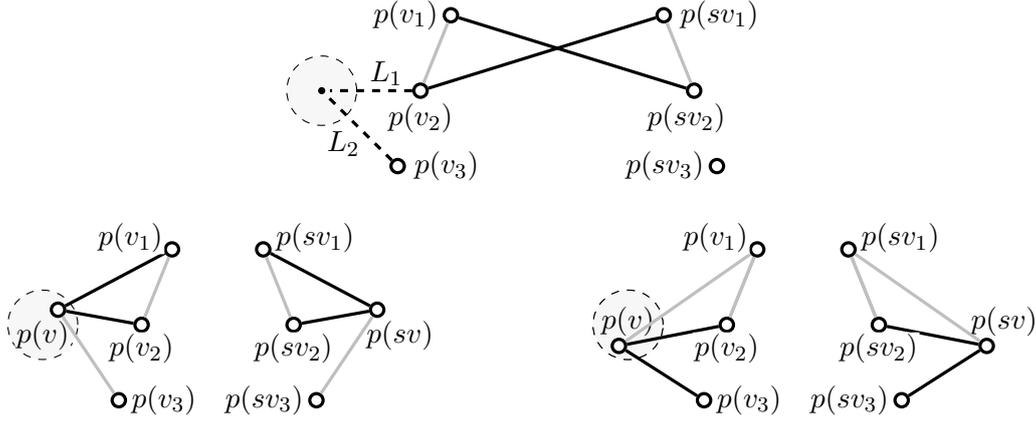
\begin{figure}[htp]
\begin{center}
\begin{tikzpicture}[very thick,scale=1]
\tikzstyle{every node}=[circle, draw=black, fill=white, inner sep=0pt, minimum width=5pt];
\node (p1) at (-1.8,0) {};
\node (p2) at (1.8,0) {};
\node (p3) at (-1.4,1) {};
\node (p4) at (1.4,1) {};
\node (p6) at (-2.1,-1) {};
\node (p7) at (2.1,-1) {};
\node (p8) at (-3.1,0) {};
\filldraw[fill=black!03!white, draw=black, thin, dashed](p8)circle(0.46cm);

\draw[dashed] (p6)--(p8);
\draw[dashed] (p1)--(p8);

\draw[lightgray](p1)--(p3);
\draw(p1)--(p4);
\draw[lightgray](p2)--(p4);
\draw(p2)--(p3);

\node[rectangle, draw=white,fill=white,right](l1) at (1.6,1) {\small $p(sv_1)$};
\node[rectangle, draw=white,fill=white,below](l2) at (1.7,-0.15) {\small $p(sv_2)$};
\node[rectangle, draw=white,fill=white,left](l3) at (-1.55,1) {\small$p(v_1)$};
\node[rectangle, draw=white,fill=white,below](l4) at (-1.8,-0.15) {\small$p(v_2)$};
\node[rectangle, draw=white,fill=white,right](l3) at (-1.9,-1) {\small$p(v_3)$};
\node[rectangle, draw=white,fill=white,left](l4) at (1.95,-1) {\small$p(sv_3)$};

\filldraw[fill=black, draw=black, thin](p8)circle(0.04cm);

\node[rectangle, draw=white,fill=white,above](l4) at (-2.25,0.05) {\small $L_1$};
\node[rectangle, draw=white,fill=white, left](l4) at (-2.58,-0.7) {\small $L_2$};
\end{tikzpicture}
\\~\\
\begin{tikzpicture}[very thick,scale=1]
\tikzstyle{every node}=[circle, draw=black, fill=white, inner sep=0pt, minimum width=5pt];

\node (p1) at (-1,0) {};
\node (p2) at (1,0) {};
\node (p3) at (-0.6,1) {};
\node (p4) at (0.6,1) {};
\node (p6) at (-1.3,-1) {};
\node (p7) at (1.3,-1) {};
\node (p8) at (-2.3,0) {};
\filldraw[fill=black!03!white, draw=black, thin, dashed](p8)circle(0.46cm);
\node (p9) at (-2.1,0.2) {};
\node (p10) at (2.1,0.2) {};

\node[rectangle, draw=black!03!white,fill=black!03!white,below](l4) at (-2.3,0.07) {\small $p(v)$};
\node[rectangle, draw=white,fill=white,below](l4) at (2.4,0.07) {\small$p(sv)$};

\draw[lightgray](p1)--(p3);
\draw[lightgray](p2)--(p4);
\draw[lightgray](p6)--(p9);
\draw(p3)--(p9);
\draw(p1)--(p9);
\draw[lightgray](p7)--(p10);
\draw(p4)--(p10);
\draw(p2)--(p10);

\node[rectangle, draw=white,fill=white,below](l1) at (-1,-0.1) {\small $p(v_2)$};
\node[rectangle, draw=white,fill=white,below](l2) at (1,-0.1) {\small $p(sv_2)$};
\node[rectangle, draw=white,fill=white,left](l3) at (-0.7,1.15) {\small$p(v_1)$};
\node[rectangle, draw=white,fill=white,right](l4) at (0.75,1.15) {\small$p(sv_1)$};
\node[rectangle, draw=white,fill=white,right](l3) at (-1.15,-1) {\small$p(v_3)$};
\node[rectangle, draw=white,fill=white,left](l4) at (1.15,-1) {\small$p(sv_3)$};
\end{tikzpicture}
\hspace{1.8cm}
\begin{tikzpicture}[very thick,scale=1]
\tikzstyle{every node}=[circle, draw=black, fill=white, inner sep=0pt, minimum width=5pt];
\filldraw[fill=black!03!white, draw=black, thin, dashed](p8)circle(0.46cm);

\node (p1) at (-1,0) {};
\node (p2) at (1,0) {};
\node (p3) at (-0.6,1) {};
\node (p4) at (0.6,1) {};
\node (p6) at (-1.3,-1) {};
\node (p7) at (1.3,-1) {};
\node (p8) at (-2.3,0) {};
\filldraw[fill=black!03!white, draw=black, thin, dashed](p8)circle(0.46cm);
\node (p9) at (-2.42,-0.28) {};
\node (p10) at (2.42,-0.28) {};

\node[rectangle, draw=black!03!white,fill=black!03!white,below](l4) at (-2.3,0.26) {\small $p(v)$};
\node[rectangle, draw=white,fill=white,below](l4) at (2.65,0.26) {\small$p(sv)$};

\draw[lightgray](p1)--(p3);
\draw[lightgray](p2)--(p4);
\draw[lightgray](p1)--(p3);
\draw[lightgray](p2)--(p4);
\draw(p6)--(p9);
\draw[lightgray](p3)--(p9);
\draw(p1)--(p9);
\draw(p7)--(p10);
\draw[lightgray](p4)--(p10);
\draw(p2)--(p10);

\node[rectangle, draw=white,fill=white,below](l1) at (-1,-0.1) {\small $p(v_2)$};
\node[rectangle, draw=white,fill=white,below](l2) at (1,-0.1) {\small $p(sv_2)$};
\node[rectangle, draw=white,fill=white,left](l3) at (-0.7,1.15) {\small$p(v_1)$};
\node[rectangle, draw=white,fill=white,right](l4) at (0.75,1.15) {\small$p(sv_1)$};
\node[rectangle, draw=white,fill=white,right](l3) at (-1.15,-1) {\small$p(v_3)$};
\node[rectangle, draw=white,fill=white,left](l4) at (1.15,-1) {\small$p(sv_3)$};

\end{tikzpicture}
\end{center}
\caption{\emph{Constructing a placement for a modified $(\bZ_2,\theta)$ $1$-extension.}}
\label{Aug2ExtPlacement}
\end{figure}

{\bf Case $(2)$}. Suppose $P$ joins $v_3$ to $sv_2$.
Let $a\in \bR^2$ be the point of intersection of the lines $L_1=\{p(v_1)+tx_1:t\in \bR\}$ and $L_2=\{p(v_3)+ty_1:t\in \bR\}$ and let $B(a,r)$ be the open  ball with centre $a$ and radius $r>0$.
Arguments similar to Case $(1)$ may now be applied.

In both cases, the induced monochrome subgraphs $G_{F_1}$ and $G_{F_2}$ are spanning trees in $G$.
Thus, by Theorem \ref{SpanningTreeThm}, $(G,p)$ is isostatic. 

\endproof

\begin{proposition}
\label{RigidityPreservation3}
Suppose there exists a graph $H$ with the following properties.
\begin{enumerate}[(i)]
\item $G$ is obtained from $H$ by a $(\bZ_2,\theta)$ fixed-vertex-to-$W_5$ extension.
\item There exists a faithful representation $\tau:\bZ_2\to \GL(\bR^2)$ and  a point $p_H$ such that  $(H,p_H)$ is a well-positioned and isostatic  bar-joint  framework in $(\mathbb{R}^2,\|\cdot\|_\P)$ which is $\bZ_2$-symmetric with respect to $\theta$ and $\tau$, where $s$ preserves the facets of $\P$.
\end{enumerate}
Then there exists $p$ such that $(G,p)$ is a well-positioned and isostatic bar-joint framework in $(\mathbb{R}^2,\|\cdot\|_\P)$ which is $\bZ_2$-symmetric with respect to $\theta$ and $\tau$. 
\end{proposition}

\proof
Suppose $G$ is  obtained by a $(\bZ_2,\theta)$ fixed-vertex-to-$W_5$ extension on the vertex $v_0$. Then $W_5$ is a subgraph of $G$ which is $\bZ_2$-symmetric with respect to $\theta$ and $v_0$ is the unique vertex in $W_5$ which is fixed by $s$. 
Also, $V(G)=V(H)\cup V(W_5)$ and $V(H)\cap V(W_5)=\{v_0\}$.
Let $N(v_0)$ be the set of vertices in $H$ which are adjacent to $v_0$.
Let $B(p_H(v_0),r)$ be the open  ball with centre $p_H(v_0)$ and radius $r>0$.
There exists a $\bZ_2$-symmetric and isostatic placement $p^*$ of the graph $W_5$ as shown in Lemma \ref{W5Placement}.
By translating this placement it may be assumed that $p^*(v_0)=p_H(v_0)$.
By contracting this placement it may be further assumed that $p^*(V(W_5))$  is contained in $B(p_H(v_0),r)$.
Define $p$ as follows: Let $p(v)=p_H(v)$ for all $v\in V(H)$ and let $p(v)=p^*(v)$ for all $v\in V(W_5)$.
If $r$ is sufficiently small then given any edge $v_0w\in E(H)$, the  framework colouring of $wv_0$ induced by $p_H$ is the same as the framework colouring of the reassigned edge $wv_j$ induced by $p$.
Note that  $(G,p)$ is well-positioned and $\bZ_2$-symmetric.
Moreover, $G_{F_1}$ and $G_{F_2}$ are spanning trees in $G$
and so $(G,p)$ is isostatic, by Theorem \ref{SpanningTreeThm}. 
\endproof

\subsection{Proof of the main result.}

\begin{proof}[Proof of Theorem \ref{CsThm}]
$(A)$ $(i)\Rightarrow (ii)$ 
If  $(G,p)$ is well-positioned then $G$ is an edge-disjoint union of the monochrome subgraphs $G_{F_1}$ and $G_{F_2}$ induced by $p$. If $(G,p)$ is also isostatic then, by Theorem \ref{SpanningTreeThm}, $G_{F_1}$ and $G_{F_2}$ are spanning trees in $G$. By Lemma \ref{Compatible}, the monochrome subgraphs $G_{F_1}$ and $G_{F_2}$ are $\bZ_2$-symmetric with respect to $\theta$.
By \cite[Corollary~4.2 (ii)]{kit-sch}, no edge of $G$ is fixed by the reflection $s$.

$(ii)\Rightarrow (i)$
By hypothesis, $(G,\theta)$ is an admissible pair and so there exists a construction chain of graphs and allowable extensions as shown in Theorem \ref{ConstructionScheme},
 $W_5=G^1\to G^2 \to\cdots \to G^n = G$.
The restriction of $\theta$ to $V(W_5)$ is the unique action $\theta^*$.
Choose a representation $\tau:\bZ_2\to\GL(\bR^2)$ such that $s$ is a reflection and preserves the facets of $\P$.  
Then the base graph $W_5$  has a well-positioned and isostatic placement $p^*$ in $(\bR^2,\|\cdot\|_\P)$, described in Lemma \ref{W5Placement},
which is $\bZ_2$-symmetric with respect to $\theta$ and $\tau$.
By Propositions \ref{RigidityPreservation1} - \ref{RigidityPreservation3}, for each subsequent graph $G^k$ in the construction chain there must exist a placement $p_k$ such that $(G^k,p_k)$ is  well-positioned and isostatic  
in $(\bR^2,\|\cdot\|_\P)$ and $\bZ_2$-symmetric with respect to $\theta$ and $\tau$. In particular,  such a placement must exist for $G$.

$(B)$ The proof of the implication $(i)\Rightarrow (ii)$ is analogous to $(A)$. A key difference in proving the converse is that graphs satisfying $(ii)$ may have any number of fixed vertices. To obtain a construction scheme similar to $(A)$,  extend the class of simple graphs in $(ii)$ by allowing two parallel edges between pairs of fixed vertices. The construction scheme includes three different types of $0$-extension: addition of a fixed vertex together with either two parallel edges or two non-parallel edges, and addition of a symmetric pair of vertices of degree two. The construction scheme also includes the $1$-extension, modified $1$-extension and vertex-to-$W_5$ extension. Two additional base graphs are required. The first one is the graph $G_1$ depicted in Fig.~\ref{fig:symisofw}(B), and the other one is the trivial graph $K_1$ consisting of a single vertex and no edge. An additional bridging move is also needed. This move takes two graphs $G_1$ and $G_2$ with $|E(G_i)|=2|V(G_i)|-2$ for $i=1,2$, and creates the graph $G_1\oplus G_2$
which has vertex set the disjoint union of $V(G_1)$ and $V(G_2)$ and edge set $E(G_1)\cup
E(G_2) \cup \{e, f\}$ where each of the edges $e$ and $f$ connects a vertex in $G_1$ with a vertex in $G_2$. 

By a standard argument, $G$ must contain a vertex of degree $2$ or $3$.
If $G$ contains a vertex of degree $2$ then this can be removed by an inverse $0$-extension. If $G$ contains a vertex of degree $3$ with at least two fixed neighbours then this can be removed by an inverse $1$-extension which places parallel edges between two of the fixed neighbours.
In the remaining case a simple counting argument can be applied to show that $G$ contains a contractible copy of $W_5$, or a copy of $G_1$ that can be removed by an inverse bridging move, or  a vertex of degree 3 
which can be removed by an inverse $1$-extension or modified $1$-extension. The existence of geometric placements which preserve isostaticity is established in a similar manner to (A). The notion of an induced framework colouring can be extended by regarding an edge which is not well-positioned as having two distinct framework colours. This convention provides a natural correspondence between induced framework colourings and graphs with two parallel edges between fixed vertices. Note that only intermediate graphs in the construction chain will have parallel edges and so the final framework will be well-positioned.

$(C)$ $(i)\Rightarrow (ii)$ 
By the argument in $(A)$, $G$ is an edge-disjoint union of the monochrome subgraphs $G_{F_1}$ and $G_{F_2}$ and these subgraphs are $\bZ_2$-symmetric with respect to $\theta$. 
By \cite[Corollary~4.4 (ii)]{kit-sch}, either no edge or two edges of $G$ are fixed by the half-turn rotation $s$. 

$(ii)\Rightarrow (i)$
Let $\tau:\bZ_2\to\GL(\bR^2)$  be the representation for which $\tau(s)$ acts as  rotation by $\pi$ about the origin.
Then $s$ is a half-turn rotation which preserves the facets of $\P$. 
If no edge of $G$ is fixed by $s$ then proceed as in the proof of $(A)$.
If two edges of $G$ are fixed by $s$ then let $e=v_1v_2$ and $f=v_3v_4$ be these
fixed edges.
Construct a new graph $\hat{G}$ as follows: Let $V(\hat{G})=V(G)\cup \{w_0\}$ where $w_0\notin V(G)$.
Let $E(\hat{G})=(E(G)\backslash\{e,f\})\cup \{v_1w_0,\,v_2w_0,\,v_3w_0,\,v_4w_0\}$.
Extend the action of $\theta$ on $G$ to $\hat{G}$ by defining $\hat{\theta}:\bZ_2\to\Aut(\hat{G})$ with $\hat{\theta}(s)v=\theta(s)v$ for all $v\in V(G)$ and $\hat{\theta}(s)w_0=w_0$.
Let $G_1$ and $G_2$ be  edge-disjoint spanning trees with union $G$,  both of which are $\bZ_2$-symmetric with respect to $\theta$, and assume without loss of generality that $e\in E(G_1)$.
Define $\hat{G}_1$ and $\hat{G}_2$ by setting $\hat{G}_1=(G_1\backslash\{e\})\cup \{v_1w_0,v_2w_0\}$ and $\hat{G}_2=(G_2\backslash\{f\})\cup \{v_3w_0,v_4w_0\}$.
Then $\hat{G}_1$ and $\hat{G}_2$ are edge-disjoint spanning trees in $\hat{G}$ which are $\bZ_2$-symmetric with respect to $\hat{\theta}$ and $G$ is the union of $\hat{G}_1$ and $\hat{G}_2$.
Note that no edge of $\hat{G}$ is fixed by $s$ and so, by the above argument, there exists $\hat{p}$ such that $(\hat{G},\hat{p})$ is well-positioned and isostatic  in $(\bR^2,\|\cdot\|_\P)$ and $\bZ_2$-symmetric with respect to $\theta$ and $\tau$.
Define $p(v)=\hat{p}(v)$ for all $v\in V(G)$.
Then $(G,p)$  is well-positioned and isostatic in $(\bR^2,\|\cdot\|_\P)$ and $\bZ_2$-symmetric with respect to $\theta$ and $\tau$. 

$(D)$ The method of proof for $(D)$ is a modification of the methods in $(A)$.  See Figure~\ref{fig:2basewheel} for an illustration of a 1-extension and modified 1-extension applied to the base graph $W_5$.
\end{proof}

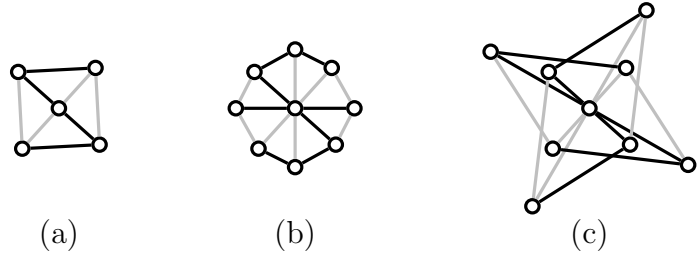
\begin{figure}[htp]
\begin{center}
\begin{tikzpicture}[very thick,scale=0.6]
\tikzstyle{every node}=[circle, draw=black, fill=white, inner sep=0pt, minimum width=5pt];

 \node (v1) at (0:0cm) {};
 \node (v2) at (48:1.2cm) {};
 \node (v3) at (138:1.2cm) {};
 \node (v4) at (228:1.2cm) {};
 \node (v5) at (318:1.2cm) {};

\draw[lightgray] (v1)  --  (v2);
\draw (v1)  --  (v3);
\draw[lightgray] (v1)  --  (v4);
\draw (v1)  --  (v5);

\draw (v2)  --  (v3);
\draw[lightgray] (v3)  --  (v4);
\draw (v5)  --  (v4);
\draw[lightgray] (v5)  --  (v2);

\node[rectangle, draw=white,fill=white](l1) at (0,-2.8) {(a)};
\end{tikzpicture} 
\hspace{1.3cm}
\begin{tikzpicture}[very thick,scale=0.6]

\tikzstyle{every node}=[circle, draw=black, fill=white, inner sep=0pt, minimum width=5pt];

 \node (v1) at (0:0cm) {};

 \node (v2) at (48:1.2cm) {};
 \node (v3) at (138:1.2cm) {};
 \node (v4) at (228:1.2cm) {};
 \node (v5) at (318:1.2cm) {};

 \node (v22) at (90:1.3cm) {};
 \node (v33) at (180:1.3cm) {};
 \node (v44) at (270:1.3cm) {};
 \node (v55) at (360:1.3cm) {};

\draw[lightgray] (v1)  --  (v2);
\draw (v1)  --  (v3);
\draw[lightgray] (v1)  --  (v4);
\draw (v1)  --  (v5);

\draw[lightgray] (v1)  --  (v22);
\draw (v1)  --  (v33);
\draw[lightgray] (v1)  --  (v44);
\draw (v1)  --  (v55);

\draw[lightgray] (v3)  --  (v33);
\draw (v44)  --  (v4);
\draw[lightgray] (v33)  --  (v4);
\draw (v44)  --  (v5);

\draw[lightgray] (v55)  --  (v5);
\draw (v22)  --  (v2);
\draw[lightgray] (v55)  --  (v2);
\draw (v22)  --  (v3);
\node[rectangle, draw=white,fill=white](l1) at (0,-2.8) {(b)};
\end{tikzpicture} 
\hspace{1.3cm}
\begin{tikzpicture}[very thick,scale=0.6]

\tikzstyle{every node}=[circle, draw=black, fill=white, inner sep=0pt, minimum width=5pt];

 \node (v1) at (0:0cm) {};

 \node (v2) at (48:1.2cm) {};
 \node (v3) at (138:1.2cm) {};
 \node (v4) at (228:1.2cm) {};
 \node (v5) at (318:1.2cm) {};

 \node (v22) at (60:2.5cm) {};
 \node (v33) at (150:2.5cm) {};
 \node (v44) at (240:2.5cm) {};
 \node (v55) at (330:2.5cm) {};

\draw[lightgray] (v1)  --  (v2);
\draw (v1)  --  (v3);
\draw[lightgray] (v1)  --  (v4);
\draw (v1)  --  (v5);

\draw[lightgray] (v1)  --  (v22);
\draw (v1)  --  (v33);
\draw[lightgray] (v1)  --  (v44);
\draw (v1)  --  (v55);

\draw (v2)  --  (v33);
\draw[lightgray] (v3)  --  (v44);
\draw (v4)  --  (v55);
\draw[lightgray] (v5)  --  (v22);

\draw (v22)  --  (v3);
\draw[lightgray] (v33)  --  (v4);
\draw (v44)  --  (v5);
\draw[lightgray] (v55)  --  (v2);
\node[rectangle, draw=white,fill=white](l1) at (0,-2.8) {(c)};

\end{tikzpicture} 

     \end{center}
\vspace{-0.3cm}
     \caption{(a) An isostatic placement of the base graph $W_5$ for $\C_4$.  (b),(c) $\C_4$-symmetric frameworks obtained by applying a 1-extension and a modified 1-extension respectively to the base graph.	The edges of the `anti-symmetric' monochrome spanning trees are shown in  gray and black.}
\label{fig:2basewheel}
\end{figure}

Note that, analogously to the Euclidean case \cite{schtan}, an infinitesimally rigid $\bZ_2$- or $\bZ_4$-symmetric framework in $(\bR^2,\|\cdot\|_\P)$ does not necessarily have a spanning isostatic subframework with the same symmetry. For example, it is easy to construct infinitesimally rigid frameworks with $\mathcal{C}_s$-symmetry in $(\bR^2,\|\cdot\|_\P)$ (where the reflection preserves the facets of $\P$) which have no vertex or more than one vertex that is fixed by the reflection (see Figure~\ref{infrigex}). Thus, for symmetric frameworks in $(\bR^2,\|\cdot\|_\P)$,   infinitesimal rigidity can in general not be characterised  in terms of symmetric isostatic subframeworks. Therefore, a more advanced approach (along the lines of \cite{schtan}, for example) is needed to analyse symmetric over-braced frameworks for infinitesimal rigidity. 
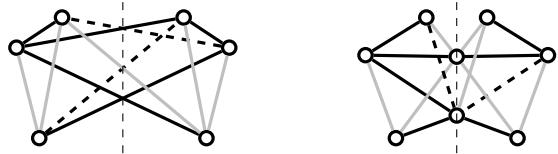
\begin{figure}[htp]
\begin{center}
\begin{tikzpicture}[very thick,scale=1]
\tikzstyle{every node}=[circle, draw=black, fill=white, inner sep=0pt, minimum width=5pt];
  \path (-1.4,0.3) node (p1)  {} ;
    \path (1.4,0.3) node (p2)  {} ;
    \path (-0.8,0.7) node (p3)  {} ;
     \path (0.8,0.7) node (p4)  {} ;
     \path (-1.1,-0.9) node (p5)  {} ;
     \path (1.1,-0.9) node (p6)  {} ;
\draw (p1)  --  (p3);
        \draw [lightgray] (p1)  --  (p5);
        \draw (p1)  --  (p4);
        \draw (p6)  --  (p1);

        \draw (p2)  --  (p4);
        \draw [lightgray](p2)  --  (p6);
        \draw [dashed](p2)  --  (p3);
        \draw (p5)  --  (p2);

        \draw[lightgray] (p5)  --  (p3);
        \draw [lightgray](p4)  --  (p6);
        
          \draw [dashed](p5)  --  (p4);
        \draw [lightgray](p3)  --  (p6);
        \draw[dashed,thin] (0,-1.1)  --  (0,1);
        
\end{tikzpicture} 
\hspace{1.3cm}
\begin{tikzpicture}[very thick,scale=1]
\tikzstyle{every node}=[circle, draw=black, fill=white, inner sep=0pt, minimum width=5pt];
  \path (-1.2,0.2) node (p1)  {} ;
    \path (1.2,0.2) node (p2)  {} ;
    \path (-0.4,0.7) node (p3)  {} ;
     \path (0.4,0.7) node (p4)  {} ;
     \path (-0.8,-0.9) node (p5)  {} ;
     \path (0.8,-0.9) node (p6)  {} ;
     \path (0,0.18) node (p7)  {} ;
     \path (0,-0.6) node (p8)  {} ;
     
\draw (p1)  --  (p3);
        \draw [lightgray] (p1)  --  (p5);
        \draw (p1)  --  (p7);
        \draw (p1)  --  (p8);
        
        \draw (p2)  --  (p4);
        \draw  [lightgray](p2)  --  (p6);
        \draw (p2)  --  (p7);
        \draw [dashed](p2)  --  (p8);
        
         \draw [lightgray](p3)  --  (p7);
        \draw [lightgray](p4)  --  (p7);
        \draw [lightgray](p5)  --  (p7);
        \draw [lightgray](p6)  --  (p7);
        \draw [dashed](p3)  --  (p8);
        \draw  [lightgray](p4)  --  (p8);
        \draw (p5)  --  (p8);
        \draw (p6)  --  (p8);
        \draw[dashed,thin] (0,-1.1)  --  (0,1);\end{tikzpicture} 
     \end{center}
     \caption{Examples of $\mathcal{C}_s$-symmetric frameworks in $(\bR^2,\|\cdot\|_\infty)$ which are infinitesimally rigid with $|V_s|=0$ and $|V_s|=2$, respectively, but do not contain a spanning isostatic subframework with the same symmetry. The edges of the monochrome spanning trees are shown in gray and black; the dashed edges do not belong to any of the two trees.}
\label{infrigex}
\end{figure}




\bibliographystyle{abbrv}
\def\lfhook#1{\setbox0=\hbox{#1}{\ooalign{\hidewidth
  \lower1.5ex\hbox{'}\hidewidth\crcr\unhbox0}}}
	
\end{document}